\documentclass[10pt]{amsart} 

\usepackage[english]{babel}
\usepackage[utf8]{inputenc}
\usepackage[T1]{fontenc}

\usepackage{amsmath}
\usepackage{amssymb}
\usepackage{amsfonts}
\usepackage{amssymb}
\usepackage{amsthm}
\usepackage{amscd}
\usepackage{csquotes}

\usepackage[hidelinks]{hyperref}
\usepackage{geometry}

\newcommand{\R}{\setsymbol{R}}%

\newtheorem{definition}{Definition}[section]
\newtheorem{lemma}[definition]{Lemma}
\newtheorem{theorem}[definition]{Theorem}
\newtheorem{proposition}[definition]{Proposition}
\newtheorem{corollary}[definition]{Corollary}
\newtheorem{conjecture}[definition]{Conjecture}

\newtheorem{remark}[definition]{Remark}

\DeclareMathOperator{\Ric}{Ric}

\usepackage{slashed}
\usepackage{amsmath,amsthm,verbatim,amssymb,amsfonts,amscd, graphicx}
\usepackage{graphics}
\usepackage{tensor}
\usepackage{enumitem}
\usepackage[utf8]{inputenc}
\usepackage{graphicx}
\usepackage{xcolor}
\usepackage{hyperref}
\usepackage{subcaption}
\usepackage{float}

\definecolor{darkblue}{rgb}{0.31, 0.12, 0.59}
\definecolor{darkorange}{rgb}{0.78, 0.31, 0.08}

\def\Ric{\operatorname{Ric}}

\def\R{\mathbb{R}}
\def\C{\mathbb{C}}

\def\tr{\operatorname{tr}}

\def\div{\operatorname{div}}

\def\m{\mathfrak m}

\author{Sven Hirsch}
\address{Columbia University, 2990 Broadway, New York NY 10027, USA}
\email{sven.hirsch@columbia.edu }

\title{Hawking Mass Monotonicity for Initial Data Sets}

\begin{document}
 
 \maketitle

 \begin{abstract}
We introduce new systems of PDE on initial data sets $(M,g,k)$ whose solutions model double-null foliations.
This allows us to generalize Geroch's monotonicity formula for the Hawking mass under inverse mean curvature flow to initial data sets satisfying the dominant energy condition. 
We study the existence theory of these systems and give geometric applications.
 \end{abstract}

 \section{Introduction}

 In the pioneering work \cite{HuiskenIlmanen} G.~Huisken and T.~Ilmanen proved the Riemannian Penrose inequality for connected horizons.
 Their proof relies on the monotonicity of the Hawking mass under inverse mean curvature flow (IMCF) which has been first discovered by R.~Geroch, P.~Jang and R.~Wald in \cite{Geroch, JangWald}.
 Due to IMCF's crucial role, it has been of great interest to generalize IMCF to the spacetime ($k\ne0$) setting.
 However, so far all approaches either lacked an existence theory \cite{BrayKhuri, BHMS, Frauendiener} or a monotonicity formula \cite{Moore, HuiskenWolf}, and more than half a century after Penrose's seminal paper \cite{Penrose}, the general Penrose conjecture still remains wide open.
 In this paper we propose a new method to generalize IMCF which is the first instance where there are both monotonicity and partial existence results, cf. Theorem \ref{T:PDE}.

\begin{theorem}\label{Thm Intro}
Consider a smooth initial data set $(M^3,g,k)$ with boundary $\partial M=\partial_+M\cup\partial_-M$.
Let $(u,v)$ be a solution to spacetime IMCF with appropriate\footnote{See Section \ref{S:Penrose} for a precise statement and definitions.} boundary and regularity conditions. Then
    \begin{align}\label{eq:cor main monotonicity}
  m_H(\partial_+M)-m_H(\partial_-M)
  \ge&\int_M( \mu(|\nabla u|+|\nabla v|)+\langle J,\nabla u-\nabla v\rangle)d\mu
\end{align}
where
\begin{align*}
    m_H(\Sigma)=\sqrt{\frac{|\Sigma|}{16\pi}}\left(1 - \frac1{16\pi}\int_{\Sigma}\theta_+\theta_-dA\right)
\end{align*}
is the spacetime Hawking mass, $\theta_\pm=H\pm \tr_\Sigma (k)$ are the future and past null expansions, and
 \begin{align*}
     \mu=&\frac12(R+(\tr_g(k))^2-|k|^2),\quad\quad\quad J=\div(k-\tr_g(k)g)
 \end{align*}
 are the energy, and momentum densities.
\end{theorem}

If the dominant energy condition (DEC) $\mu\ge|J|$ holds, the RHS of equation \eqref{eq:cor main monotonicity} is non-negative.
Equation \eqref{eq:cor main monotonicity} follows from a more general divergence identity which applies to a one-parameter family of PDE systems.
These systems simultaneously generalize IMCF, spacetime harmonic, and $p$-harmonic functions with their corresponding monotonicity formulas.
We are able to solve some of these systems in full generality, and the IMCF system in spherical symmetry.
For any slice $(M^3,g,k)$ in the Minkowski spacetime the solution $(u,v)$ to any of these systems is given by the restricting the null-functions $u=r+t$ and $v=r-t$ to $(M^3,g,k)$.

\subsection{Statement of results}

Our work relies on the following divergence identity:

 \begin{theorem}\label{T:main}
Let $a\in[0,1]$ and suppose $u,v\in C^{2,\alpha}(M)$ are positive solutions of the system
\begin{align}\label{system}
\begin{split}
\overline\Delta_+u=&a(\overline\nabla^2_+)_{\nu_u\nu_u}u\\
\overline\Delta_-v=&a(\overline\nabla^2_-)_{\nu_v\nu_v}v
\end{split}
\end{align}
with $|\nabla u|,|\nabla v|\ne0$. Here $\nu_u=\frac{\nabla u}{|\nabla u|}$ is the unit normal to the level-sets of $u$, $\overline\Delta_\pm=\tr_g\overline\nabla^2_\pm$ and 
\begin{align*}
    (\overline\nabla^2_+)_{ij}u=\nabla_{ij}u+k_{ij}|\nabla u|-\frac{|\nabla u||\nabla v|+\langle \nabla u,\nabla v\rangle}{u+v}g_{ij}+\frac{\nabla_iu\nabla_jv+\nabla_ju\nabla_iv}{u+v},\\
   (\overline\nabla^2_-)_{ij}v=\nabla_{ij}v-k_{ij}|\nabla v|-\frac{|\nabla u||\nabla v|+\langle \nabla u,\nabla v\rangle}{u+v}g_{ij}+\frac{\nabla_iu\nabla_jv+\nabla_ju\nabla_iv}{u+v}.
\end{align*}
Then
\begin{align}\label{integral formula}
\begin{split}
\div Y=&\frac{|\overline\nabla_+^2u|^2-(a(\overline\nabla_+^2)_{\nu_u\nu_u}u)^2}{|\nabla u|}+\frac{|\overline\nabla_+^2v|^2-(a(\overline\nabla_-^2)_{\nu_v\nu_v}v)^2}{|\nabla v|}\\
&+2\mu(|\nabla u|+|\nabla v|)+2\langle J,\nabla u-\nabla v\rangle\\
&-2K_u|\nabla u|-2K_v|\nabla v|
\end{split}
\end{align}
where $K_u,K_v$ are the Gaussian curvatures of the level sets of $u,v$, and
\begin{align*}
Y=&2\nabla (|\nabla u|+|\nabla v|)+2k(\nabla (u-v),\cdot)+4(|\nabla u|\nabla v+|\nabla v|\nabla u)\frac1{u+v}\\
&-2\Delta u\frac{\nabla u}{|\nabla u|}-2\Delta v\frac{\nabla v}{|\nabla v|}-2\tr_g(k)\nabla (u-v).
\end{align*}
\end{theorem}

The first line of the RHS of equation \eqref{integral formula} is always non-negative, the second line \eqref{integral formula} is non-negative in case the DEC is satisfied, and the third line can be controlled via Gauss-Bonnet's theorem upon integration.

Formula \eqref{integral formula} generalizes many important results:

 In case $k=0$, we can set $u=v$ and system \eqref{system} decouples and the PDE system simplifies to 
 \begin{align}\label{PDE k=0}
     \Delta u=a\nabla_{\nu_u\nu_u}u+2\frac{|\nabla u|^2}u.
 \end{align}
 For $a=1$, the function $u$ solving \eqref{PDE k=0} is rescaled IMCF and integrating the divergence identity \eqref{integral formula} yields the Hawking mass monotonicity formula.
 For $k=0$ and $0\le a<1$, $u$ is a rescaled $p$-harmonic function.
The corresponding divergence identity \eqref{integral formula} has been first discovered by V.~Agostiniani, L.~Mazzieri and F.~Oronzio in \cite{AMO}.
This led to a new proof of the Riemannian positive mass theorem (PMT), and, together with C.~Mantegazza, the Riemannian Penrose inequality \cite{AMMO}.
However, even in the special case $k=0$, the above formula has some new content since we can prescribe different boundary conditions for $u$ and $v$, such that the system does not decouple.

Another special case is given by $v=0$.
Then $u$ is a spacetime harmonic function, i.e. $u$ solves the PDE $\Delta u=-\tr_g(k)|\nabla u|$, and integrating equation \eqref{integral formula} recovers the main formula of \cite[Proposition 3.2]{HKK} which led to a proof of the spacetime PMT by the author, D.~Kazaras, and M.~Khuri.

Moreover, in spherical symmetry, equation \eqref{integral formula} recovers the monotonicity formula of the spacetime Hawking mass \cite{Hayard} which implies the Penrose inequality in this setting.


The motivation for the PDE systems \eqref{system} is three-fold. 
First, they are geometrically natural and obtained by studying the null functions $u=r+t$ and $v=r-t$ in Minkowski space.
Second, they have a simple analytic structure. 
In particular there are no second-order coupling terms enabling us to show existence for $a=0$ in full generality.
Third, the use of two functions instead of one, mimics the use of two-component spinors, cf. \cite{Witten, HirschZhang2}.

\begin{theorem}\label{T:Minkowski}
Consider the functions $u=r+t$ and $v=r-t$ in Minkowski space $\R^{3,1}$.
Then the restrictions of $u,v$ to any initial data set $(M,g,k)\subset\R^{3,1}$ solve system \eqref{system} for ever $a\in[0,1]$.
Moreover, we have $(\overline\nabla^2_+)_{ij}u=0$ and $(\overline \nabla^2_-)_{ij}v=0$.
\end{theorem}

\begin{theorem}\label{T:PDE}
Let $(M,g,k)$ be a smooth initial data set.
Suppose that the boundary of $M$ has two connected components $\partial_-M$ and $\partial_+M$.
Then we can solve system \eqref{system} for $a=0$, i.e. for any positive constants $c_-,c_+,d_-,d_+$, there exist positive functions $u,v\in C^{2,\alpha}(M)$ solving
\begin{align*}
\Delta u=&-\tr_g(k)|\nabla u|+\frac{3|\nabla u||\nabla v|+\langle \nabla u,\nabla v\rangle}{u+v},\\
\Delta v=&\tr_g(k)|\nabla v|+\frac{3|\nabla u||\nabla v|+\langle \nabla u,\nabla v\rangle}{u+v}
\end{align*}
on $M$, with Dirichlet boundary data $u=c_\pm$, $v=d_\pm$ on $\partial_\pm M$.
\end{theorem}

For $a=1$ we expect the existence theory to be more complicated due to the degenerate elliptic nature of the underlying equation. Fortunately, the coupling is still only to first order in this case.
This contrasts other approaches to spacetime IMCF where the underlying PDE systems are more complicated \cite{BrayKhuri, Frauendiener, Jaracz}.
In particular, in our formulation no third-order equations, non-local equations, or backward parabolic equations appear.

\subsection{Organization of the paper}
In Section \ref{S:Motivation} we motivate the PDE systems \eqref{system} in detail, compare it to spinors, and prove Theorem \ref{T:Minkowski}.
Next, we show Theorem \ref{T:main} in Section \ref{S:integral formula proof}.
Additionally, we demonstrate that the same ideas also apply to manifolds equipped with an electrical field and prove in this setting another divergence identity, Theorem \ref{T:charge}.
In Section \ref{S:Penrose} we begin with providing an overview of the Penrose conjecture and survey previous approaches towards solving it. 
Next, we establish Theorem \ref{Thm Intro} and the Penrose inequality in spherical symmetry which both follow from Theorem \ref{T:main}.
Finally, we prove Theorem \ref{T:PDE} in Section \ref{S:PDE}.

\textbf{Acknowledgements.}  
This work was supported in part by the National Science Foundation under Grant No. DMS-1926686 and by the IAS School of Mathematics.
The author would like to thank Hubert~Bray, Simon~Brendle, Demetre~Kazaras, Marcus~Khuri and Yiyue~Zhang for insightful discussions and for their interest in this work.
The author also grateful to the anonymous referee whose suggestions lead to various improvements.


\section{Motivating the PDE systems}\label{S:Motivation}

\subsection{A new perspective on IMCF in the Riemannian setting}\label{new perpsective}

 Let $(M,g)$ be an asymptotically flat, 3-dimensional, complete manifold with non-negative scalar curvature $R$.
 Such manifolds arise naturally in General Relativity (GR) where they are used to model isolated gravitational systems such as stars, galaxies and black holes.
 The apparent horizon of the latter can be modeled by a connected, outermost, minimal surface $\Sigma_0 \subset M$.
We denote with $\Sigma_t$ the IMCF starting from $\Sigma_0$, i.e. unless there are jumps, $\Sigma_t$ flows in outward normal direction with speed $\frac1{H_t}$ where $H_t$ is the mean curvature of $\Sigma_t$.
The famous monotonicity formula for the Hawking mass under IMCF \cite{Geroch, JangWald, HuiskenIlmanen} states that $\partial_t\m_H(\Sigma_t)\ge0$ where
 \begin{align*}
     \m_H(\Sigma)=\sqrt{\frac{|\Sigma|}{16\pi}}\left( 1-\frac1{16\pi}\int_\Sigma H^2dA\right).
 \end{align*}
 An important ingredient in G.~Huisken and T.~Ilmanen's proof of the Riemannian Penrose inequality is to recognize that there is a level-set formulation of IMCF for which one can find weak solutions.
 More precisely, by defining the function $U$ via $\Sigma_t=\partial\{x\in M :U(x)<t\}$, we see that $U$ satisfies the degenerate elliptic equation 
 \begin{align*}
     \div \left(\frac{\nabla U}{|\nabla U|}\right)=|\nabla U|,
 \end{align*}
 where we note that the term on the left hand side equals the mean curvature of the level-sets $\Sigma_t$.
 Reparametrizing $u=e^{\frac12U}$, we obtain the homogeneous equation
 \begin{align}\label{IMCF rescaled}
     \Delta u=\nabla^2_{\nu\nu}u+2\frac{|\nabla u|^2}u
 \end{align}
 where $\nu$ is the outer normal to the level sets $\Sigma_t$.
 In this context, we can rephrase the Hawking mass \emph{monotonicity formula} $\m_H(\Sigma_t)-\m_H(\Sigma_0)\ge0$, $t\ge0$, as integral formula
 \begin{align}\label{mH monotonicity}
  \m_H(\Sigma_t)-\m_H(\Sigma_0)=\frac1{16\pi}\int_{\Omega_t}\left(R|\nabla u|+  \frac{|\mathcal H^2u|^2-(\mathcal H^2_{\nu\nu}u)^2}{|\nabla u|}     \right)dV.
 \end{align}
Here $\Omega_t$ is the region bounded by $\Sigma_0$ and $\Sigma_t$, and $\mathcal H$ is a symmetric 2-tensor defined by
 \begin{align}\label{eq:Hessian IMCF}
     \mathcal H_{ij}u=\nabla_{ij}u-\frac{|\nabla u|^2}ug_{ij}+\frac{\nabla_iu\nabla_ju}{u}.
 \end{align}
The RHS of equation \eqref{mH monotonicity} is non-negative in case $R\ge0$.
 Next, we will take a more general point of view.
In case we do not integrate the integrand on the RHS of equation \eqref{mH monotonicity} over a domain $\Omega$, we obtain the \emph{divergence identity}
\begin{align}\label{mH monotonicity2}
    R|\nabla u|+  \frac{|\mathcal H^2u|^2-(\mathcal H^2_{\nu\nu}u)^2}{|\nabla u|} -2K_u|\nabla u|  = 2\div \left(\nabla |\nabla u|+\frac{|\nabla u|}u\nabla u-\Delta u\frac{\nabla u}{|\nabla u|}\right)
\end{align}
 where $K_u$ is the Gaussian curvature of $\Sigma_t=\{u(x)=t\}$.
 This formulation of the Hawking mass monotonicity naturally generalizes to spacetime setting, cf. Theorem \ref{T:main}.

 Similarly, given a $(2-a)$-harmonic function $U$, the function $u=U^{-\frac{1-a}{1+a}}$ satisfies the homogeneous equation $\Delta u=\nabla_{\nu\nu}u+2\frac{|\nabla u|^2}u$.
 In this setting, V.~Agostiani's, L.~Mazzieri's and F.~Oronzio's monotonicity formula \cite{AMO} takes the form.
 \begin{align}\label{AMO formula}
    R|\nabla u|+  \frac{|\mathcal H^2u|^2-(a\mathcal H^2_{\nu\nu}u)^2}{|\nabla u|} -2K_u|\nabla u|  = 2\div \left(\nabla |\nabla u|+\frac{|\nabla u|}u\nabla u-\Delta u\frac{\nabla u}{|\nabla u|}\right).
\end{align}
We remark that in both cases $|\nabla u|$ is asymptotic to $1$ and equals the static potential in the case of equality.

\subsection{Double null foliations and slices in Minkowski space}

To give a new proof of the spacetime PMT, we introduced together with D.~Kazaras, M.~Khuri in \cite{HKK} \emph{spacetime harmonic} functions which are functions satisfying the PDE $\Delta u=-\tr_g(k)|\nabla u|$.
Subsequently, we established with Y.~Zhang the corresponding rigidity and we refer to \cite{HKKZ} for a detailed overview of spacetime harmonic functions.
In case $(M,g,k)$ arises as a subset of Minkowski space $\R^{3,1}$, the spacetime harmonic function $u$ can be obtained by restricting a null coordinate function of Minkowski space such as $x+t$, to $(M,g,k)$.
Hence, in the case of equality of the spacetime PMT the level-sets $\Sigma_t$ of $u$ can be obtained by intersecting null planes with the initial data set $(M,g,k)\subset \R^{3,1}$.

A similar situation occurs for any $a\in[0,1]$ for our system \eqref{system}.
However, instead of leading to a single null foliation, the level sets $\Sigma_u,\Sigma_v$ of $u,v$ lead to a \textit{double null foliation}, cf. Theorem \ref{T:Minkowski}.

A similar situation occurs in the spinorial setting:
An $SL(2,\C)$ Weyl spinor gives rise to a single null vector field while a Dirac spinor gives rise to two null vector fields \cite[page 352]{Wald}.
This connection will be elaborated upon in a forthcoming paper.

\begin{figure}[H]
\includegraphics[totalheight=7cm]{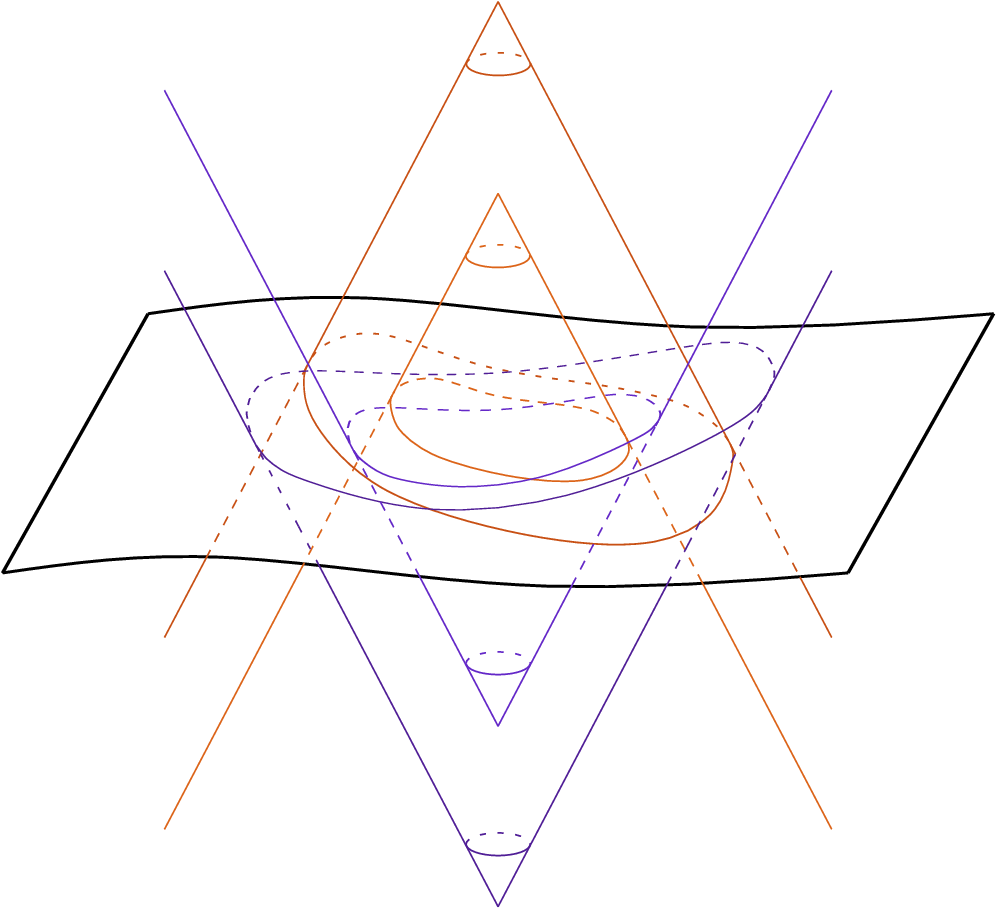}
\caption{
The double null foliation (\textcolor{darkorange}{$\Sigma_u$},\textcolor{darkblue}{$\Sigma_v$}) for the initial data set $(M,g,k)\subset \R^{3,1}$ is obtained by intersecting past and future directed lightcones in $\R^{3,1}$ with $(M,g,k)$.
}
\end{figure}

\begin{figure}[H]
    \begin{subfigure}{0.4\textwidth}
    \includegraphics[width=\textwidth]{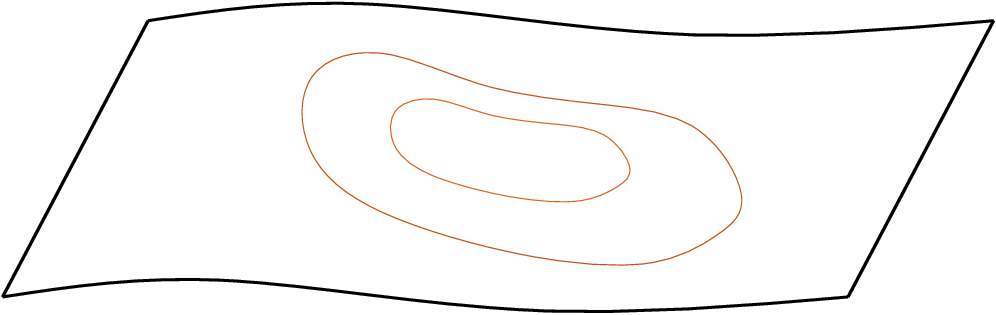}
    \caption*{$({M,g,k})\subset\R^{3,1}$ with \textcolor{darkorange}{$\Sigma_u$} Foliation}
    \label{fig:first}
\end{subfigure}
\begin{subfigure}{0.4\textwidth}
    \includegraphics[width=\textwidth]{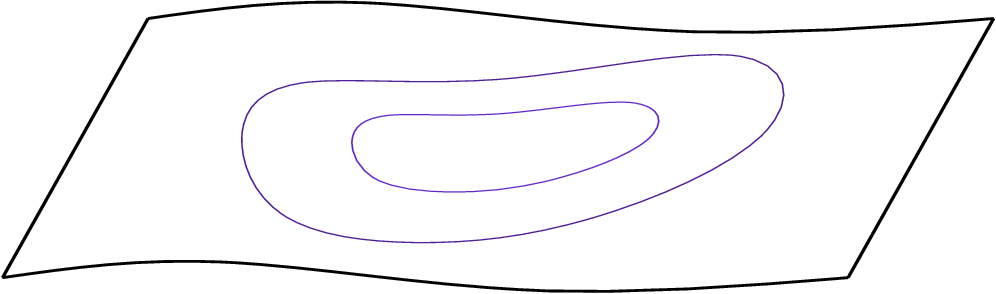}
    \caption*{$({M,g,k})\subset\R^{3,1}$ with \textcolor{darkblue}{$\Sigma_v$} Foliation}
    \label{fig:second}
\end{subfigure}
\caption{There are no monotone quantities associated with the level sets \textcolor{darkorange}{$\Sigma_u$} and \textcolor{darkblue}{$\Sigma_v$} individually. Hence, a divergence identity as in Theorem \ref{T:main} is a more general concept than a monotonicity formula.}
\end{figure}

Next, we show that for any initial data set $(M,g,k)$ contained in Minkowski space $(\R^{3,1},\hat g)$ the functions $u=r+t$ and $v=r-t$ solve system \eqref{system} for every $a\in[0,1]$.

\begin{proof}[Proof of Theorem \ref{T:Minkowski}]
The key observation is that $\hat\nabla^2(v-u)=0$ and $\hat\nabla^2(uv)=2\hat g$.
This implies
\begin{align*}
   2\hat g=& v\hat\nabla^2u+u\hat \nabla^2v+\hat \nabla u\otimes \hat \nabla v+\hat \nabla v\otimes \hat\nabla u
   =(u+v)\hat\nabla^2u+\hat \nabla u\otimes \hat \nabla v+\hat\nabla v\otimes \hat \nabla u.
\end{align*}
Restricting this identity onto $T^\ast M\otimes T^\ast M$ we obtain
\begin{align*}
     2 g=&(u+v)\hat\nabla^2|_{T^\ast M\otimes T^\ast M}u+ \nabla u\otimes \nabla v+ \nabla v\otimes \nabla u.
\end{align*}
Next, let us denote with $N$ the future pointing unit normal of $M\subset \R^{3,1}$.
Since $\hat \nabla u$ and $\hat \nabla v$ are null, 
\begin{align*}
    \hat\nabla^2|_{T^\ast M\otimes T^\ast M}u=&\nabla^2u+kN(u)=\nabla^2u+k|\nabla u|.
\end{align*}
Moreover, 
\begin{align*}
    2=\hat g(\hat \nabla u,\hat \nabla v)=\langle \nabla u,\nabla v\rangle+\hat g(N(u)N,N(v)N)=\langle \nabla u,\nabla v\rangle+|\nabla u||\nabla v|.
\end{align*}
Combining everything, yields $\overline \nabla^2_+u=0$.
Similarly, one can show $ \overline \nabla^2_-v=0$.
\end{proof}

Another example of a double null foliation will be given in Section \ref{S:integral formula charge} in the context of the spacetime PMT with charge.

\subsection{Analytic advantages}\label{SS:analytic}

Why is it necessary to use two functions to formulate spacetime IMCF instead of just a single function? 
Similarly, why is it necessary to have a pair of spacetime $p$-Green's functions?
To answer these questions we will focus on system \eqref{system} for the case $a=0$.

To prove the Riemannian PMT V.~Agostiani, L.~Mazzieri and F.~Oronzio used a new monotonicity formula for the Green's function \eqref{AMO formula}, i.e. solutions of the PDE $\Delta u=2\frac{|\nabla u|^2}u$.
Previously, H.~Bray, M.~Khuri, D.~Kazaras, and D.~Stern \cite{BKKS} used asymptotically linear harmonic functions based on D.~Stern's intergral identity \cite{Stern}.
In a joint work with M.~Khuri and D.~Kazaras, we \enquote{spacetime-ized} the Laplace equation and the corresponding integral formula which led to a proof of the spacetime PMT \cite{HKK}.
To \enquote{spacetime-ize} the Laplace equation, we replaced the Laplace equation $\Delta u=0 $ with the \enquote{spacetime Laplace} equation $\Delta u=-\tr_gk|\nabla u|$.
It now appears natural to consider the PDE $\Delta u=-\tr_g(k)|\nabla u|+2\frac{|\nabla u|^2}u$ to \enquote{spacetime-ize} Agostiani-Mazzieri-Oronzio's monotonicity formula.
However, this naive approach is flawed.

This contrasts spinor methods where such straightforward combination of PDEs usually works.
For instance, to prove the spacetime PMT one solves the Dirac equation $\slashed D\psi=-\frac12\tr_g(k)e_0\psi$, to prove the PMT with charge one solves $\slashed D\psi=Ee_0\psi$, and for the spacetime PMT with charge one solves $\slashed D\psi=Ee_0\psi-\frac12\tr_g(k)e_0\psi$ \cite[Theorem 11.9]{BartnikChrusciel}.
Using level-sets the situation becomes more complicated and one is again led to a system of two PDEs, cf. Theorem \ref{T:charge}.

There are two complications with this above approach.
First, this PDE does not characterize slices in Minkowski spacetime. 
Neither $u=r+t$ nor any other function solves the equation $\Delta u=-\tr_g(k)|\nabla u|+2\frac{|\nabla u|^2}u$ for every IDS $(M,g,k)$ in $\R^{3,1}$.
However this is a necessary property since slices of Minkowski space arise as case of equality of the PMT.

Second, when attempting to obtain a divergence identity there are algebraic obstacles involving cross terms. 
More precisely, bad terms of the form $-\tr_g(k)\frac{|\nabla u|^2}u$ are left over in the computation.
Similarly, the PDE  $\Delta v=+\tr_g(k)|\nabla v|+2\frac{|\nabla v|^2}u$ leads to the bad term $+\tr_g(k)\frac{|\nabla v|^2}v$.
Thus, if we could just add both divergence identities the bad terms would cancel perfectly as long as $\frac{|\nabla u|^2}u$ and $\frac{|\nabla v|^2}v$ coincide.
This is made rigorous precisely by coupling the PDE which ensure that the factors in front of the $\tr_g(k)$ terms coincide.

A similiar situation takes place for $a=1$. In this case (rescaled) null IMCF (with speed $\frac1{\theta_+}$ and $\frac1{\theta_-}$) takes the form
  \begin{align*}
     \theta_+|\nabla u|=&2\frac{|\nabla u|^2}u,\\
    \theta_-|\nabla v|=& 2\frac{|\nabla v|^2}v.
 \end{align*}
Coupling these PDE via
  \begin{align*}
     \theta_+|\nabla u|=&2\frac{|\nabla u||\nabla v|+\langle \nabla u,\nabla v\rangle}{u+v},\\
    \theta_-|\nabla v|=& 2\frac{|\nabla u||\nabla v|+\langle \nabla u,\nabla v\rangle}{u+v}.
 \end{align*}
gives system \eqref{system} for $a=1$.
Again the crossterms will cancel in the computation of the Hawking mass monotonicity formula.



\section{Proof of the divergence identities}\label{S:integral formula proof}

In this section we prove Theorem \ref{T:main} as well as a divergence identity for manifolds equipped with an additional vector field.
Even though, the underlying PDE appear much more complicated compared to their Riemannian analogs, the computations are surprisingly short.

\begin{proof}[Proof of Theorem \ref{T:main}]
Introducing the notation
\begin{align*}
    \begin{split}
        \operatorname{I}=&4\Delta u\frac{|\nabla v|}{u+v}+4\nabla_{\nu_u\nu_v}u\frac{|\nabla v|}{u+v}-4\frac{|\nabla v|(|\nabla u||\nabla v|+\langle \nabla u,\nabla v\rangle)}{(u+v)^2},\\
        \operatorname{II}=&4\Delta v\frac{|\nabla u|}{u+v}+4\nabla_{\nu_u\nu_v}v\frac{|\nabla u|}{u+v}-4\frac{|\nabla u|(|\nabla u||\nabla v|+\langle \nabla u,\nabla v\rangle)}{(u+v)^2},
    \end{split}
\end{align*}
we have
\begin{align}\label{a}
    &4\operatorname{div}\left(\frac{|\nabla u|\nabla v+|\nabla v|\nabla u}{u+v}\right)=\operatorname{I}+\operatorname{II}.
\end{align}
Next, we recall 
    \begin{align*}
        \overline \nabla^2_+u=\nabla^2 u+k|\nabla u|-\frac{|\nabla u||\nabla v|+\langle \nabla u,\nabla v\rangle}{u+v}g+\frac{\nabla u\otimes \nabla v+\nabla v\otimes \nabla u}{u+v}
    \end{align*}
which implies
    \begin{align}\label{b}
    \begin{split}
        &|\overline{\nabla}^2_+u|^2\\=&
|\nabla^2u|^2+\frac{5|\nabla u|^2|\nabla v|^2+2|\nabla u||\nabla v|\langle \nabla u,\nabla v\rangle+\langle \nabla u,\nabla v\rangle^2}{(u+v)^2}\\
&-2\Delta u\frac{|\nabla u||\nabla v|+\langle \nabla u,\nabla v\rangle}{u+v}+4\nabla_{\nu_u\nu_v}u\frac{|\nabla u||\nabla v|}{u+v}\\
&+|k|^2|\nabla u|^2+2k_{ij}|\nabla u|\nabla_{ij}u -|\nabla u|A     
    \end{split}
    \end{align}
where 
\begin{align*}
    A= 2k_{ij}\left(   \frac{|\nabla u||\nabla v|+\langle \nabla u,\nabla v\rangle}{u+v}g_{ij}-\frac{\nabla_iu\nabla_jv+\nabla_ju\nabla_iv}{u+v}      \right)   .
\end{align*}
Moreover, 
\begin{align}\label{c}
    \begin{split}
        &(\overline{\Delta}_+u-a(\overline\nabla^2_+)_{\nu_u\nu_u}u)^2\\
        =&(\Delta u)^2-2\Delta u\frac{3|\nabla u||\nabla v|+\langle \nabla u,\nabla v\rangle}{u+v}+\frac{9|\nabla u|^2|\nabla v|^2+6|\nabla u||\nabla v|\langle\nabla u,\nabla v\rangle+\langle \nabla u,\nabla v\rangle^2}{(u+v)^2}\\
        &-(a(\overline\nabla^2_+)_{\nu_u\nu_u}u)^2-2a(\overline \Delta_+u-a(\overline\nabla^2_+)_{\nu_u\nu_u}u)(\overline\nabla^2_+)_{\nu_u\nu_u}u\\
        &+\operatorname{tr}_g(k)^2|\nabla u|^2+2\tr_g(k)|\nabla u|\Delta u-|\nabla u|B
    \end{split}
\end{align}
where
\begin{align*}
    B= 2\operatorname{tr}_g(k)\frac{3|\nabla u||\nabla v|+\langle \nabla u,\nabla v\rangle}{u+v}.
\end{align*}
Combining equations \eqref{a},\eqref{b} and \eqref{c} with the PDE $\overline \Delta_+u=a(\overline\nabla^2_+)_{\nu_u\nu_u}u$ yields
\begin{align}\label{1234}
    \begin{split}
   & \frac{|\overline{\nabla}^2_+u|^2-(a(\overline\nabla^2_+)_{\nu_u\nu_u}u)^2}{|\nabla u|}-\operatorname{I}\\
   =&\frac{|\nabla^2u|^2-(\Delta u)^2}{|\nabla u|}+|k|^2|\nabla u|+2k_{ij}\nabla_{ij}u-\operatorname{tr}_g(k)^2|\nabla u|-2\tr_g(k)\Delta u-A+B.
    \end{split}
\end{align}
Next, recall Stern's identity
\begin{align}\label{3}
2\div\left(\nabla |\nabla u|-\Delta u\frac{\nabla u}{|\nabla u|}\right)=\frac1{|\nabla u|}(|\nabla^2u|^2+|\nabla u|^2(R-2K_u)-(\Delta u)^2),
\end{align}
cf. (4.8) of \cite{BHKKZ}. Consequently,
\begin{align}\label{eq12345}
\begin{split}
&\frac{|\nabla^2u|^2-(\Delta u)^2}{|\nabla u|}+|k|^2|\nabla u|+2k_{ij}\nabla_{ij}u-\operatorname{tr}_g(k)^2|\nabla u|-2\tr_g(k)\Delta u\\
=&-2\mu|\nabla u|-2\langle J,\nabla u\rangle+2K_u|\nabla u|+2\div\left(\nabla |\nabla u|-\Delta u\frac{\nabla u}{|\nabla u|}+k(\nabla u,\cdot)-\tr_g(k)\nabla u\right).
\end{split}
\end{align}
Similarly, 
\begin{align}\label{123456}
    \begin{split}
    &\frac{|\overline{\nabla}^2_+v|^2-(a(\overline\nabla^2_-)_{\nu_v\nu_v}v)^2}{|\nabla v|}-\operatorname{II}\\
    =&\frac{|\nabla^2v|^2-(\Delta v)^2}{|\nabla v|}+|k|^2|\nabla v|-2k_{ij}\nabla_{ij}v-\operatorname{tr}_g(k)^2|\nabla v|+2\tr_g(k)\Delta v+A-B
    \end{split}
\end{align}
where we emphasize the opposite signs in front of $A$ and $B$ compared to equation \eqref{1234}.
Combining equations \eqref{a}, \eqref{1234}, \eqref{eq12345} and \eqref{123456}, the result follows.
\end{proof}

\begin{remark}\label{remark1}
    Observe that we never used any properties of the term $(a(\overline\nabla^2_+)_{\nu_u\nu_u}u)^2$ in the above computation.
    Indeed, we can replace this term by an arbitrary function $f_1$ which may or may not depend on $u,v$.
    More precisely, if $u,v$ solve $\overline \Delta_+u=f_1$ and $\overline\Delta_-v=f_2$, then
    \begin{align*}
\begin{split}
\div Y=&\frac{|\overline\nabla_+^2u|^2-f_1^2}{|\nabla u|}+\frac{|\overline\nabla_+^2v|^2-f_2^2}{|\nabla v|}\\
&+2\mu(|\nabla u|+|\nabla v|)+2\langle J,\nabla u-\nabla v\rangle\\
&-2K_u|\nabla u|-2K_v|\nabla v|.
\end{split}
\end{align*}
This leads to new monotonicity formulas even in the Riemannian setting.
\end{remark}

\begin{remark}
    For applications it is necessary to obtain a version of Theorem \ref{T:main} where $|\nabla u|$ is allowed to vanish.
    We expect the corresponding modifications to be straightforward and identical to the ones in the Riemannian setting, cf. \cite{AMMO, HKK, Stern}.
\end{remark}

\begin{remark}
Besides the monotonicity of the Hawking mass, another important feature of (Riemannian) IMCF is its exponential area growth of its level sets.
This is equivalent to $u^{-2}|\Sigma_u|^2$ being constant for rescaled IMCF $\Delta u=\nabla_{\nu_u\nu_u}u+2\frac{|\nabla u|^2}u$.
In case $k=\xi g$ for some function $\xi$, it is easy to see that $(u+v)^{-2}|\Sigma|^2$ is constant for all $\Sigma$ which are a level-set for both $u$ and $v$.
More precisely, for any two such surfaces $\Sigma_1$ and $\Sigma_2$ enclosing a domain $\Omega$ we have
\begin{align*}
   \left( \int_{\Sigma_2}-\int_{\Sigma_1}\right)\frac{2}{(u+v)^2}dA=&
   \left( \int_{\Sigma_2}-\int_{\Sigma_1}\right)\frac{1}{(u+v)^2}\nu\cdot\left(\frac{\nabla u}{|\nabla u|}+\frac{\nabla v}{|\nabla v|}\right)dA=\int_{\Omega}\frac{k_{\nu_v\nu_v}-k_{\nu_u\nu_u}}{(u+v)^2}dA=0.
\end{align*}
For slices of Minkowksi space the last equality still holds even without the assumption $k=\xi g$.
\end{remark}


\subsection{Spacetime charged harmonic functions}\label{S:integral formula charge}

In \cite{HKK} the spacetime PMT has been proven via spacetime harmonic functions, and in \cite{BHKKZ} the PMT with charge has been proven via \emph{charged harmonic} functions, i.e. functions solving the PDE $\Delta u=\langle E,\nabla u\rangle$ where $E$ is the electrical field.
Naively, the PDE $\Delta u=\langle E,\nabla u\rangle-\tr(k)|\nabla u|$ should lead to a proof of the spacetime PMT with charge.
However, as explained in Section \ref{SS:analytic} it is impossible due to the presence of bad cross terms between $E$ and $k$ appearing in the divergence identity.
Again, a single PDE is not sufficient and we need to consider a PDE system which should be compared to the spinorial approach in \cite[Theorem 11.9]{BartnikChrusciel}.

Given $u,v$, we set $\eta=\frac{\nu_u+\nu_v}{|\nu_u+\nu_v|}$ in case $\nu_u\ne\nu_v$ and $\eta=0$ otherwise.

 \begin{theorem}\label{T:charge}
 Let $E$ be a divergence free vector field on an IDS $(M,g,k)$.
Suppose $u,v$ solve the system
\begin{align}\label{system charged}
\begin{split}
    \Delta u=\xi E_\eta-\tr_g(k)|\nabla u|\\
    \Delta v= \xi E_\eta+\tr_g(k)|\nabla v|
    \end{split}
\end{align}
with $|\nabla u|,|\nabla v|\ne0$, and where $\xi=\sqrt{|\nabla v||\nabla u|}$. 
Then we have
\begin{align}\label{integral formula charged}
\begin{split}
\div(Z)=&\frac1{2|\nabla u|}(    |\mathcal E_+^2u|^2  +|\nabla u|^2(2\mu-2K_u-2|E|^2)+2|\nabla u|\langle J,\nabla u\rangle)\\
&+\frac1{2|\nabla v|}(    |\mathcal E^2_-v|^2  +|\nabla v|^2(2\mu-2K_v-2|E|^2)-2|\nabla v|\langle J,\nabla v\rangle).
\end{split}
\end{align}
where $K_u,K_v$ are the Gaussian curvatures of the level-sets of $u,v$, 
\begin{align*}
    Z=&\nabla |\nabla u|-\Delta u\frac{\nabla u}{|\nabla u|}+\nabla |\nabla v|-\Delta v\frac{\nabla v}{|\nabla v|}+2\xi^{-1}(|\nabla u||\nabla v|+\langle \nabla u,\nabla v\rangle)E\\
    &-\tr_g(k)\nabla u+\tr_g(k)\nabla v+k(\nabla u,\cdot)-k(\nabla v,\cdot),
\end{align*}
and 
\begin{align*}
  (  \mathcal E^2_+)_{ij} u
    =&\nabla^2_{ij}u+\xi\eta_iE_j+\xi\eta_jE_i-\xi E_\eta g_{ij}+k_{ij}|\nabla u|,\\
      (  \mathcal E^2_-)_{ij} v
    =&\nabla^2_{ij}v+\xi \eta_iE_j+\xi \eta_jE_i-\xi E_\eta g_{ij}-k_{ij}|\nabla v|.
    \end{align*}
\end{theorem}

Observe that the above formula recovers Proposition 3.2 of \cite{HKK} in case $E=0$ which has been the main ingredient in the proof of the spacetime PMT, and equation (8.7) of \cite{BHKKZ} in case $k=0$ which has been the main ingredient in the proof of the PMT with charge.
If $\mu\ge|J|+|E|$, the RHS of equation \eqref{integral formula charged} can be estimated by the Gaussian curvature terms which can be controlled with the help of Gauss-Bonnet's theorem upon integration.

\begin{proof}
The divergence identity \eqref{integral formula charged} in Theorem \ref{T:charge} reduces to the one for spacetime harmonic functions in case $\eta=0$, cf. \cite[Proposition 3.2]{HKK}.
Therefore, we assume without loss of generality that $\nu_u\ne-\nu_v$.
We compute 
\begin{align*}
    |\mathcal E^2_+u|^2
     =&|\nabla^2u|^2+2\xi^2|E|^2+4\xi\nabla_{ij}u\eta_iE_j-\xi^2E_\eta^2\\
    &+2(\xi\eta_iE_j+\xi\eta_jE_i-\xi E_\eta g_{ij})k_{ij}|\nabla u|\\
    &+|k|^2|\nabla u|^2+2\nabla^2_{ij}uk_{ij}|\nabla u|+2\tr_g(k)|\nabla u|\xi E_\eta
\end{align*}
and
\begin{align*}
    |\mathcal E^2_-v|^2
     =&|\nabla^2v|^2+2\xi^2|E|^2+4\xi\nabla_{ij}v\eta_iE_j-\xi^2E_\eta^2\\
    &-2(\xi\eta_iE_j+\xi\eta_jE_i-\xi E_\eta g_{ij})k_{ij}|\nabla v|\\
    &+|k|^2|\nabla u|^2-2\nabla^2_{ij}uk_{ij}|\nabla u|-2\tr_g(k)|\nabla u|\xi E_\eta.
\end{align*}
Combining these identities with equation \eqref{3} yields
\begin{align*}
&\div\left(\nabla |\nabla u|-\Delta u\frac{\nabla u}{|\nabla u|}+\nabla |\nabla v|-\Delta v\frac{\nabla v}{|\nabla v|}\right)
\\=&\frac1{2|\nabla u|}(    |\mathcal E^2_+u|^2-2\xi^2|E|^2-4\xi\nabla^2_{ij}u\eta_iE_j    )\\
&+\frac1{2|\nabla u|}(|\nabla u|^2(R_M-2K_u)+2\tr_g(k)\Delta u|\nabla u|+(\tr_g(k)^2-|k|^2)|\nabla u|^2-2\nabla^2_{ij}uk_{ij}|\nabla u|)\\
&+\frac1{2|\nabla v|}(    |\mathcal E_-^2v|^2-2\xi^{2}|E|^2-4\xi\nabla^2_{ij}v\eta_iE_j    )\\
&+\frac1{2|\nabla v|}(|\nabla v|^2(R_M-2K_v)-2\tr_g(k)\Delta v|\nabla v|+(\tr_g(k)^2-|k|^2)|\nabla v|^2+2\nabla^2_{ij}vk_{ij}|\nabla v|).
\end{align*}
Next, we compute
\begin{align*}
\frac1{|\nabla u|}\xi\nabla_{ij}^2u\eta_iE_j=&\div\left( \sqrt{\frac{|\nabla v|}{|\nabla u|}} \nabla_iu\eta_iE   \right)-\nabla_j\sqrt{\frac{|\nabla v|}{|\nabla u|}}E_j\nabla_\eta u-\sqrt{\frac{|\nabla v|}{|\nabla u|}}\nabla_iu\nabla_j\eta_iE_j,
\end{align*}
and 
\begin{align*}
\frac1{|\nabla v|}\xi\nabla_{ij}^2v\eta_iE_j=&\div\left( \sqrt{\frac{|\nabla u|}{|\nabla v|}} \nabla_iv\eta_iE   \right)-\nabla_j\sqrt{\frac{|\nabla u|}{|\nabla v|}}E_j\nabla_\eta v-\sqrt{\frac{|\nabla u|}{|\nabla v|}}\nabla_iv\nabla_j\eta_iE_j.
\end{align*}
Observe that
\begin{align*}
    &2\nabla_j\sqrt{\frac{|\nabla v|}{|\nabla u|}}E_j\nabla_\eta u+2\nabla_j\sqrt{\frac{|\nabla u|}{|\nabla v|}}E_j\nabla_\eta v\\
    =&\xi^{-1}\nabla_j|\nabla v|E_j\nabla_\eta u-\xi^{-1}|\nabla v||\nabla u|^{-1}\nabla_j|\nabla u|\nabla_\eta uE_j\\
    &+\xi^{-1}\nabla_j|\nabla u|E_j\nabla_\eta v-\xi^{-1}|\nabla u||\nabla v|^{-1}\nabla_j|\nabla v|\nabla_\eta vE_j=0
\end{align*}
where we used that $\langle \nu_u,\eta\rangle=\langle \nu_v,\eta\rangle$.
Moreover,
\begin{align*}
    &\sqrt{\frac{|\nabla v|}{|\nabla u|}}\nabla_iu\nabla_j\eta_iE_j+\sqrt{\frac{|\nabla u|}{|\nabla v|}}\nabla_iv\nabla_j\eta_iE_j\\
    =&\xi^{-1}\langle \nabla_E\eta,\nabla u |\nabla v|+\nabla v|\nabla u|\rangle \\
    =&\xi^{-1}||\nabla u|\nabla v+\nabla u|\nabla v||\langle \nabla_E\eta,\eta\rangle=0
\end{align*}
where we used that $\langle \eta,\eta\rangle=1$ and $\langle \nabla_E\eta,\eta\rangle=0$.
Consequently,
\begin{align*}
&\div\left(\nabla |\nabla u|-\Delta u\frac{\nabla u}{|\nabla u|}+\nabla |\nabla v|-\Delta v\frac{\nabla v}{|\nabla v|}+2\sqrt{\frac{|\nabla u|}{|\nabla v|}}\nabla_\eta vE+2\sqrt{\frac{|\nabla v|}{|\nabla u|}}\nabla_\eta uE\right)
\\=&\frac1{2|\nabla u|}(    |\mathcal E^2_+u|^2-2a\xi^2|E|^2  )\\
&+\frac1{2|\nabla u|}(|\nabla u|^2(R_M-2K_u)+2\tr_g(k)\Delta u|\nabla u|+(\tr_g(k)^2-|k|^2)|\nabla u|^2-2\nabla^2_{ij}uk_{ij}|\nabla u|)\\
&+\frac1{2|\nabla v|}(    |\mathcal E^2_-v|^2-2\xi^{2}|E|^2   )\\
&+\frac1{2|\nabla v|}(|\nabla v|^2(R_M-2K_v)-2\tr_g(k)\Delta v|\nabla v|+(\tr_g(k)^2-|k|^2)|\nabla v|^2+2\nabla^2_{ij}vk_{ij}|\nabla v|).
\end{align*}
Rearranging terms finishes the proof.
\end{proof}


\section{Spacetime IMCF in the Schwarzschild spacetime}\label{S:Penrose}

 \subsection{The Penrose Conjecture}

General Relativity (GR) is concerned with the study of Lorentzian manifolds $(\hat M^4,\hat g)$ satisfying the Einstein equations $\hat \Ric-\frac12\hat R\hat g=8\pi \hat T$ where $\hat\Ric,\hat R$ are the Ricci and scalar curvature of $\hat g$, and $\hat T$ is the stress-energy-momentum tensor.
 An interesting feature of GR is the existence of singularities which can arise even in elementary examples such as the Schwarzschild spacetime.
 However, in Schwarzschild the singularity is hidden behind the event horizon and is believed that this is generically the case\footnote{The additional assumption of genericity is necessary as Christodoulou demonstrated in \cite{Christodoulou, Christodoulou2}.} which is known as the \emph{Cosmic Censorship Conjecture}.
 To verify this conjecture, R.~Penrose proposed in 1973 a test described in Figure \ref{F:Penrose} below.

 \begin{figure}[H]
\begin{picture}(0,-100)
\put(7,168){$t$}
\put(206,27){$(M,g,k)$}
\put(206,117){Kerr}
\end{picture}
\centering
\includegraphics[totalheight=6cm]{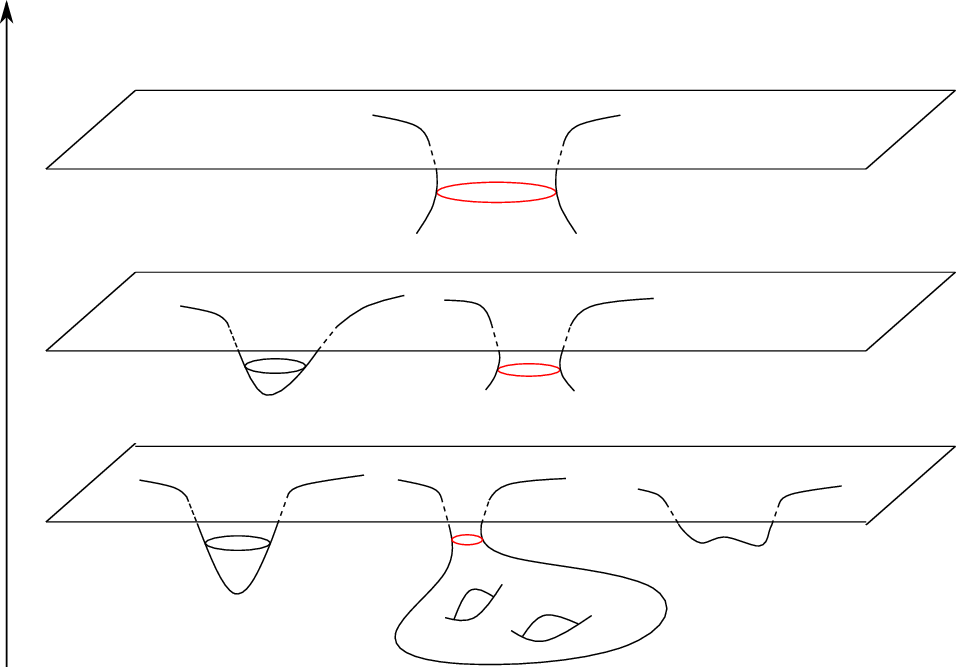}
\caption[Penrose's heuristic argument]{
Assuming the final state conjecture holds, an initial data set $(M,g,k)$ evolves to a slice of Kerr.
In Kerr $\m\ge\frac{|\Sigma|}{16\pi}$ where $\Sigma$ is the intersection of the event horizon with $(M,g,k)$.
Since matter is radiating away to infinity, and due to Hawking's area theorem \cite{Hawking, Wald}, we also have $\m\ge\frac{|\Sigma|}{16\pi}$ on $(M,g,k)$.
Combining the Cosmic Censorship with Penrose's singularity theorem \cite{Penrose2, Wald} allows us to replace the event horizon with the minimal area enclosure of the apparent horizon.
}
\label{F:Penrose}
\end{figure}
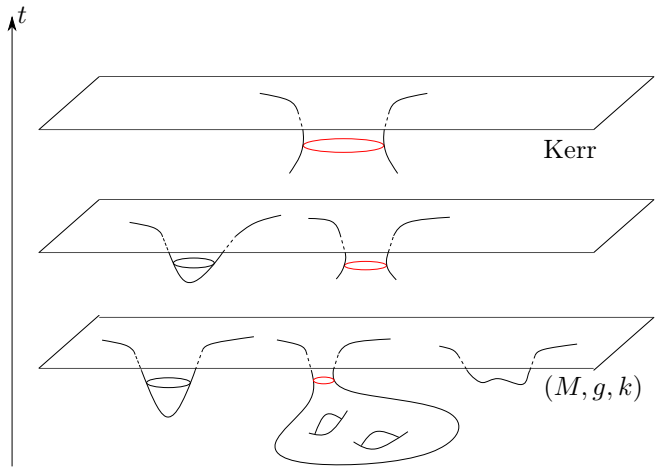

  \begin{conjecture}\label{C:Penrose}
Let $(M,g,k)$ be an initial data set satisfying the DEC. 
Let $\Sigma_0$ be a MOTS in $(M,g,k)$, and let $\Sigma$ be the minimal area enclosure of $\Sigma_0$.
Then the mass $\m=\sqrt{E^2-|P|^2}$ of $(M,g,k)$ is bounded from below by
\begin{align*}
    \m\ge \sqrt{\frac{|\Sigma|}{16\pi}}.
\end{align*}
Moreover, we have equality if and only if $(M,g,k)$ is a slice in Schwarzschild spacetime.
 \end{conjecture}
 Here a marginally outer trapped surfaces (MOTS) is a surface $\Sigma$ satisfying $\theta_+=H+\tr_{g_{\Sigma}}k=0$ and models an apparent horizons.
Moreover, the ADM energy and momentum of $(M,g,k)$ are defined by
 \begin{align*}
E=\lim_{r\rightarrow\infty}\frac{1}{16\pi}\int_{S_{r}}\sum_i \left(g_{ij,i}-g_{ii,j}\right)\upsilon^j dA,\quad\quad
P_i=\lim_{r\rightarrow\infty}\frac{1}{8\pi}\int_{S_{r}} \left(k_{ij}-(\tr_g k)g_{ij}\right)\upsilon^j dA.
\end{align*}

We refer to the books of D.~Lee and R.~Wald \cite{Lee, Wald} for an introduction to GR and to M.~Mars' survey \cite{Mars} for a detailed explanation of R.~Penrose's heuristic.

 A counter example to the Penrose Conjecture would pose a serious challenge to the Cosmic Censorship conjecture which is considered to be the weakest link in the above argument.
Besides its physical significance, the Penrose Conjecture also presents a strengthening of the famous positive mass theorem \cite{AMO, BKKS, Eichmair, EHLS, HKK, HuiskenIlmanen, Li, Miao, SY1, SY2}.
 By time-reversal, i.e. by replacing $k$ with $-k$, one also expects conjecture \ref{C:Penrose} hold also for marginally inner trapped surfaces (MITS), i.e surfaces satisfying $\theta_-=0$.

 We remark that in the statement of the Penrose inequality it is necessary to consider the minimal area enclosure $\Sigma$ instead of the MOTS $\Sigma_0$ itself.
It is easy to construct counterexamples to $m\ge\sqrt{\frac{|\Sigma_0|}{16\pi}}$, see for instance Figure 1 in \cite{HuiskenIlmanen}, and even the assumption of $\Sigma_0$ being an outermost MOTS is insufficient as demonstrated by I.~Ben-Dov in \cite{Ben-Dov}.

  The Penrose Conjecture has been established in the case $k=0$ by G.~Huisken and T.~Ilmanen \cite{HuiskenIlmanen} (for connected horizons), and by H.~Bray \cite{Bray} (for arbitrary horizons).
  H.~Bray's proof employs Bray's conformal flow and has also been generalized up to dimension 7 by H.~Bray and D.~Lee in \cite{BrayLee} and to the electrostatic setting by M.~Khuri, G.~Weinstein and S.~Yamada in \cite{KWY}.

 In the general case $k\ne0$ the conjecture is wild open outside spherical symmetry  \cite{BrayKhuri, BrayKhuri2, Hayard, Hayard2, IMM, MalecMurchadha} and H.~Roesch' result on certain null cones \cite{Roesch}.
In the pioneering work \cite{BrayKhuri} H.~Bray and M.~Khuri proposed three methods to couple Jang's equation.  This leads to several system of PDE which imply the Penrose conjecture for \textit{generalized horizons} (i.e. surfaces satisfying $\theta_+\theta_-=0$\footnote{
 Thus, there have to arise some complications in the existence theory in view of A.~Carrasco and M.~Mars' counter example \cite{CarrascoMars}.}) with $E$ instead of $\m$.

Given that system \eqref{system} with $a=1$ and integral formula \eqref{integral formula} generalizes IMCF including the Hawking mass monotonicity formula, it is natural to ask whether there are applications towards the Penrose conjecture. We obtain:

\begin{theorem}\label{T:Penrose}
Let $(M,g,k)$ be a spherically symmetric initial data set satisfying the DEC, and let $a=1$.
Then system \eqref{system} can be solved, and the integral formula \eqref{integral formula} reduces to the monotonicity formula of the spacetime Hawking energy
\begin{align*}
    \m_H(\Sigma)=\sqrt{\frac{|\Sigma|}{16\pi}}\left(1-\frac1{16\pi}\int_\Sigma\theta_+\theta_-dA\right).
\end{align*}
In particular, the Penrose inequality holds in spherical symmetry, cf. \cite{Hayard}.
\end{theorem}

One difficulty most approaches towards the Penrose conjecture face, is to solve the underlying PDE.
For instance, P.~Jang and R.~Wald already showed in \cite{JangWald} that R.~Geroch's monotonicity formula \cite{Geroch} implies the Riemannian Penrose inequality if an existence theory for IMCF can be established.
This has also been observed in the spacetime setting for \emph{Inverse Mean Curvature Vector Flow} by J.~Frauendiener \cite{Frauendiener}.
In the Riemannian case this has been resolved in \cite{AMMO, HuiskenIlmanen, Moser}, but the spacetime case this is still completely open.
Similarly, the approaches by J.~Frauendiener \cite{Frauendiener} and by H.~Bray and M.~Khuri \cite{BrayKhuri} do not yield a proof of the Penrose conjecture due to the difficulties of solving the underlying PDE systems.
Inverse mean curvature vector flow is forward-backward parabolic, the Jang-zero divergence equations are a third order system, Jang-conformal flow is non-local.
Perhaps the most promising of these PDEs is Bray-Khuri's Jang-IMCF.
However, there an existence theory is also out of reach due to the coupling of two PDEs which in their own right are already quite complicated and due to the presence of second order coupling terms.

Jang's equation has been introduced to prove the spacetime PMT \cite{SY2}. However, a simpler proof can be given via spacetime harmonic functions \cite{HKK}. 
There the underlying PDE, the spacetime Laplace equation, has only a very mild non-linearity and its existence theory is much easier than the one for Jang's equation.
Our approach of spacetime IMCF can be viewed as combining spacetime harmonic functions and IMCF.
Analytically, the systems \eqref{system} have the advantage that there are no second-order coupling terms, and there is a simple expression for $\Delta(v-u)$.
This allows us to obtain an existence theory for system \eqref{system} with $a=0$ in full generality without having to assume any symmetry, see Theorem \ref{T:PDE}, which will be proven in the next section.

\subsection{Geometric applications}

Using the divergence identity \eqref{integral formula}, we will show:
 
 \begin{corollary}\label{Cor: main}
 Let $(M,g,k)$ be an annulus satisfying the dominant energy condition with spherical boundary components $\partial_-M$ and $\partial_+M$.
 Suppose $u,v$ are constant on both $\partial_-M$ and $\partial_+M$ and that $u|_{\partial_-M}<u|_{\partial_+M}$, $v|_{\partial_-M}<v|_{\partial_+M}$.
 Then we have under the same assumptions as in Theorem \ref{T:main}
 \begin{align*}
  &(u+v)|_{\partial_+M}\left[1 - \frac1{8\pi}\int_{\partial_+M}\left(2\theta_+\frac{|\nabla u|}{u+v}+2\theta_-\frac{|\nabla v|}{u+v}-8\frac{|\nabla u||\nabla v|}{(u+v)^2}\right)dA\right]\\
  \ge&(u+v)|_{\partial_-M}\left[1 - \frac1{8\pi}\int_{\partial_-M}\left(2\theta_+\frac{|\nabla u|}{u+v}+2\theta_-\frac{|\nabla v|}{u+v}-8\frac{|\nabla u||\nabla v|}{(u+v)^2}\right)dA\right]
\end{align*}
where $\theta_\pm=H\pm \tr_{g_\Sigma}(k)$ are the null expansions.
In the case $a=1$ we furthermore have $|\nabla u|=\frac14\theta_-(u+v)$ and $|\nabla v|=\frac14\theta_+(u+v)$ on $\partial_\pm M$ which implies 
\begin{align*}
  (u+v)|_{\partial_+M}\left(1 - \frac1{16\pi}\int_{\partial_+M}\theta_-\theta_+dA\right)\ge&(u+v)|_{\partial_-M}\left(1 - \frac1{16\pi}\int_{\partial_-M}\theta_-\theta_+dA\right).
\end{align*}
 \end{corollary}
This immediately yields Theorem \ref{Thm Intro} in case one assumes $(u+v)$ is proportional to the square root of the area of $\partial_\pm M$, also see Remark \ref{remark1}.
The assumption that $\partial_\pm M$ have to be level-sets for both $u$ and $v$ appears to be restrictive at first.
However, it is impossible to obtain a monotonicity formula in the usual sense (such as the one for IMCF) for an energy functional like the Hawking mass due to the possibility of boosts.
For instance, a surface $\Sigma$ contained in an IDS $(M,g,k)$ inside Minkowski space which is a level-set for both $u=r+t$ and $v=r-t$ must have constant $t$.

We begin with the proof of Corollary \ref{Cor: main} before proceeding with the Penrose inequality in spherical symmetry, Theorem \ref{T:Penrose}, and studying arbitrary slices of the Schwarzschild spacetime.

\begin{proof}[Proof of Corollary \ref{Cor: main}]
Theorem \ref{T:main} implies
\begin{align}\label{integral formula a=1}
\begin{split}
\div Y\ge&-2K_u|\nabla u|-2K_v|\nabla v|.
\end{split}
\end{align}
where we used that $(M,g,k)$ satisfies the DEC, and where we recall
\begin{align*}
Y=&2\nabla (|\nabla u|+|\nabla v|)+2k(\nabla (u-v),\cdot)+4(|\nabla u|\nabla v+|\nabla v|\nabla u)\frac1{u+v}\\
&-2\Delta u\frac{\nabla u}{|\nabla u|}-2\Delta v\frac{\nabla v}{|\nabla v|}-2\tr_g(k)\nabla (u-v).
\end{align*}
Since both boundary components $\partial_\pm M$ are level sets for both $u$ and $v$, we have $\nu_u=\nu_v=:\nu$ on $\partial_\pm M$.
Therefore,
\begin{align*}
    Y_\nu=&2\theta_+\frac{|\nabla u|}{u+v}+2\theta_-\frac{|\nabla v|}{u+v}-8\frac{|\nabla u||\nabla v|}{(u+v)^2}.
\end{align*}
Next, we use twice the co-area formula and Gauss-Bonnet's theorem to obtain
\begin{align*}
    \int_M (K_u|\nabla u|+K_v|\nabla v|)dV=\int_{u|_{\partial_-M}}^{u|_{\partial_+M}}4\pi dt+\int_{v|_{\partial_-M}}^{v|_{\partial_+M}}4\pi dt=4\pi(u+v)|_{\partial_-M}-4\pi(u+v)|_{\partial_+M}.
\end{align*}
We emphasize how the co-area formula is applied towards two different foliations.
If $a=1$, we have additionally
\begin{align*}
\Delta u=-\tr_k|\nabla u|+\nabla_{\nu\nu}u+k_{\nu\nu}|\nabla u|+4\frac{|\nabla u||\nabla v|}{u+v}
\end{align*}
on $\partial_\pm M$.
Since $|\nabla u|\ne0$, this implies
\begin{align}\label{theta+-2}
 |\nabla v|=\frac14\theta_+(u+v).
\end{align}
Similarly, $ |\nabla u|=\frac14\theta_-(u+v)$ which finishes the proof.
\end{proof}

In order to prove Theorem \ref{T:Penrose} we first establish the following Lemma.
Again, we set $\nu=\nu_u=\nu_v$.
\begin{lemma}\label{L:Penrose}
Let $(M,g,k)$ be a spherically symmetric initial data set and let $\Sigma_0\subset M$ be the outermost apparent horizon.
Let $s$ be a smooth solution of rescaled IMCF starting from $\Sigma_0$, i.e. $\Delta s=\nabla_{\nu\nu}s+2\frac{|\nabla s|^2}s$ with $s(\Sigma_0)=\sqrt{\frac{|\Sigma_0|}{16\pi}}$.
Outside of $\Sigma_0$, we define the spherically symmetric function $w=w(r)$ via
\begin{align*}
    w(r)=\int_0^r\frac12(\tr_g(k)-k_{\nu\nu})sd\rho
\end{align*}
where $r$ is the distance to $\Sigma_0$.
Then $u,v$, implicitly defined by
\begin{align*}
    u+v=s,\quad\quad v-u=w,
\end{align*}
solve system \eqref{system} for $a=1$.
Moreover, we have
\begin{align}\label{theta+-}
 |\nabla v|=\frac14\theta_+s,\quad\quad
 |\nabla u|=\frac14\theta_-s.
\end{align}
\end{lemma}

\begin{proof}
Since $s$ solves rescaled IMCF, and using $\Delta s=\nabla_{\nu\nu}s+H\nabla_\nu s$, we deduce that $|\nabla s|=\frac12Hs$.
Moreover, we have $|\nabla w|=\frac12|\tr_g(k)-k_{\nu\nu}|s$.
Since $\Sigma_0$ is the outermost horizon and is therefore not enclosed by any MITS or MOTS, we also obtain that $\theta_+,\theta_->0$ for all spherically symmetric surfaces outside $\Sigma_0$ .
This implies $H>|\tr_g(k)-k_{\nu\nu}|$ for all spherically symmetric surfaces outside $\Sigma_0$, and since $u=\frac12(s-w)$ and $v=\frac12(s+w)$ we obtain
\begin{align*}
 |\nabla v|=\frac14\theta_+s,\quad\quad
 |\nabla u|=\frac14\theta_-s.
\end{align*}
Note that this in particular implies $\nabla u,\nabla v\ne0$ outside $\Sigma_0$ as well as $\nabla_ru,\nabla_rv>0$.
Multiplying the above identities by $|\nabla u|,|\nabla v|$, we obtain in the same fashion as in the computation of equation \eqref{theta+-2} that $(u,v)$ solve system \eqref{system} with $a=1$.
This finishes the proof.
\end{proof}

Observe that \eqref{theta+-} implies that the level sets of $u$ move by rescaled $\frac1{\theta_-}$ flow, and the level sets of $v$ move by rescaled $\frac1{\theta_+}$ flow. 
The rescaling factor is in both cases given by $\frac14s$ where $s$ is rescaled IMCF.

The above lemma immediately yields:
\begin{corollary}
We can solve system \eqref{system} for $a=1$ in spherical symmetry.
\end{corollary}

We remark that although system \eqref{system} is in many ways the most complicated for $a=1$ due to its degenerate elliptic character, the existence theory for $a=1$ is substantially simpler than for $a\in[0,1)$ in spherical symmetry.
This contrasts the Riemannian ($k=0$) setting where the existence theory for harmonic functions is elementary compared to the sophisticated existence theory for IMCF \cite{HuiskenIlmanen, Moser}.
The reason for this reverse behavior stems from the fact that the system decouples for $a=1$ in spherical symmetry as demonstrated in Lemma \ref{L:Penrose}.
However, the system appears not to decouple in spherical symmetry for $a\in [0,1)$, and the function $u+v$ is not the rescaling of a $p$-harmonic function.

\begin{proof}[Proof of Theorem \ref{T:Penrose}]
Let $(M,g,k)$ be a spherically symmetric initial data set satisfying the DEC, and let $u,v$ be solutions to system \eqref{system} for $a=1$ outside the horizon $\Sigma_0$ as described in Lemma \ref{L:Penrose}.
As in the proof of Corollary \ref{Cor: main} above, we obtain
\begin{align}\label{eq:cor}
    (u+v)|_{\Sigma_2}\left(1-\frac1{16\pi}\int_{\Sigma_2}\theta_+\theta_-dA\right)\ge     (u+v)|_{\Sigma_1}\left(1-\frac1{16\pi}\int_{\Sigma_1}\theta_+\theta_-dA\right)
\end{align}
for any spherically symmetric surface $\Sigma_2$ enclosing $\Sigma_1$ enclosing $\Sigma_1$.
Since $u+v=s$ solves rescaled IMCF $\Delta s=\nabla_{\nu\nu}s+2\frac{|\nabla s|^2}s$ with $s(\Sigma_0)=\sqrt{\frac{|\Sigma_0|}{16\pi}}$, and because IMCF is uniformly area expanding, we obtain
\begin{align*}
   \sqrt{\frac{|\Sigma_2|}{16\pi}}\left(1-\frac1{16\pi}\int_{\Sigma_2}\theta_+\theta_-dA\right)\ge      \sqrt{\frac{|\Sigma_1|}{16\pi}}\left(1-\frac1{16\pi}\int_{\Sigma_1}\theta_+\theta_-dA\right)
\end{align*}
Hence, the spacetime Hawking energy
\begin{align}
    \m_H(\Sigma_t)=\sqrt{\frac{|\Sigma_t|}{16\pi}}\left(1-\frac1{16\pi}\int_{\Sigma_t} \theta_-\theta_+dA\right)
\end{align}
is monotonically increasing for spherically symmetric initial data sets satisfying the DEC.
This monotonicity then implies the Penrose inequality in this setting, cf. \cite{Hayard}.
\end{proof}

\subsection{Slices in Schwarzschild}
The Penrose inequality can also be viewed as a characterization of the Schwarzschild spacetime.
In the Riemannian case a solution to the overdetermined system
\begin{align}\label{eq Hessian S4}
    \nabla^2_{ij}u=\frac{|\nabla u|^2}ug_{ij}+\frac{\nabla_i u \nabla_j u}{u}
\end{align}
implies that $M$ must be flat Euclidean space, see for instance \cite{AMO}.
If \eqref{eq Hessian S4} holds for all $i,j$ except for $i=j=\nu$, we gain an additional degree of freedom and $M$ must be the $t=0$ slice of Schwarzschild with some arbitrary mass $\m$, cf. \cite{HuiskenIlmanen}.
In fact, the term $|\nabla^2u-\tfrac{|\nabla u|^2}ug_{ij}+\tfrac{\nabla_i u \nabla_j u}{u}|^2-(\nabla_{\nu\nu}u)^2$ appears in the Hawking mass monotonicity formula as shown in Section \ref{new perpsective}.
Taking the trace of equation \eqref{eq Hessian S4} omitting the $(\nu,\nu)$-direction gives (rescaled) IMCF.
Similarly, a proof of the spacetime Penrose inequality is expected to classify arbitrary slices in Schwarzschild via such overdetermined equations.

In our case the situation is quite similar. 
However, instead of having null functions the underlying objects will be null vector fields.

Recall that the Schwarzschild spacetime $(\hat M,\hat g)$ of mass $\m\ge0$ in static coordinates
\begin{align*}
    \hat g=-\phi^2 dt^2+\phi^{-2}dr^2+r^2g_{S^2}
\end{align*}
where $\phi^2=(1-\frac{2\m}r)$.

\begin{theorem}
   Let $T=\phi^2 \nabla t$ and define $X=\nabla r+T$, $Y=\nabla r-T$.
    Then $X,Y,r$ solve 
    \begin{align}\label{eq schwarzschild}
    \hat\nabla_\alpha X_\beta=\hat\nabla_\alpha Y_\beta=\frac{\phi^2}r\hat g_{\alpha\beta}-\frac{X_\alpha Y_\beta+X_\beta Y_\alpha}{2r}
    \end{align}
    for all vectors $\alpha,\beta$ as long as at least one of of these vectors is orthogonal to the $r-t$ plane.
   Moreover, on each spherically symmetric IDS $(M,g,k)$ contained in Schwarzschild $ (\hat M,\hat g)$, the vector field $T$ (and hence $X$, $Y$) are integrable with $X=du$ and $Y=dv$.
\end{theorem}

\begin{proof}
This is a direct computation.
\end{proof}

We remark that on any (not necessarily spherically symmetric) initial data set
\begin{align*}
|X||Y|+\langle X,Y\rangle= \hat g(N(u)N,N(v)N)+\langle X,Y\rangle= \hat g( X,Y)=2\phi^2.
\end{align*}
Moreover $\hat \nabla_\alpha X_\beta=\nabla_\alpha X_\beta+k_{\alpha\beta}|X|$.
Hence, equation \eqref{eq schwarzschild} indeed recovers $\overline\nabla^2_+u=0$, and $\overline\nabla^2_-v=0$.

Besides the null vector fields above, there are also globally defined null functions in Schwarzschild which satisfy a system of PDEs for every slice of Schwarzschild:

\begin{proposition}
Consider the functions $u=r^\ast +t$ and $v=r^\ast-t$ in the Schwarzschild spacetime $(\hat M,\hat g)$ where $r^\ast=r+2m\log(\tfrac{r}{2m}-1)$ is the Tortoise coordinate.
Then on any IDS $(M,g,k)\subseteq (\hat M,\hat g)$ (not necessarily spherically symmetric) the restrictions of the functions $u,v$ onto $M$ satisfy
\begin{align}\label{system alternative2}
\begin{split}
\theta_+(\Sigma_u)|\nabla u|=&\phi^2 \frac{|\nabla u||\nabla v|+\langle\nabla u,\nabla v\rangle} {r},\\
\theta_-(\Sigma_v)|\nabla v|=&\phi^2\frac{|\nabla u||\nabla v|+\langle\nabla u,\nabla v\rangle} {r}.
\end{split}
\end{align}
and
\begin{align}\label{hessian alternative}
\begin{split}
\nabla^2u=&-k|\nabla u|+\frac g{2r}\phi^2(|\nabla u||\nabla v|+\langle\nabla u,\nabla v\rangle)\\
&-\frac \m{r^2}\nabla u\otimes \nabla u-\frac {\phi^2}{2r}(\nabla u\otimes \nabla v+\nabla v\otimes \nabla u),\\
\nabla^2v=&-k|\nabla v|+\frac g{2r}\phi^2(|\nabla u||\nabla v|+\langle\nabla u,\nabla v\rangle)\\
&-\frac \m{r^2}\nabla v\otimes \nabla v-\frac {\phi^2}{2r}(\nabla u\otimes \nabla v+\nabla v\otimes \nabla u).
\end{split}
\end{align}
\end{proposition}

\begin{proof}
    This is another direct computation.
\end{proof}

Finally, we would like to point out the importance of equations such as \eqref{hessian alternative} lies in the fact that they can be used to characterize slices in certain spacetimes.
See for instance \cite{HirschZhang, HKK} for slices in Minkowski space and Proposition 2 in J.~Krohn's paper \cite{Krohn} for slices of Schwarzschild.


\section{PDE existence theory}\label{S:PDE}

Throughout this section let $(M,g,k)$ be a compact $3$-dimensional initial data set whose boundary has two connected components $\partial_\pm M$, and let $c_\pm, \, d_\pm\in\mathbb R$ be positive constants 

To solve the system 
 \begin{align}\label{eq:main}
 \begin{split}
\Delta u=&-\tr_g(k)|\nabla u|+\frac{3|\nabla u||\nabla v|+\langle \nabla u,\nabla v\rangle}{u+v},\\
\Delta v=&\tr_g(k)|\nabla v|+\frac{3|\nabla u||\nabla v|+\langle \nabla u,\nabla v\rangle}{u+v}
\end{split}
\end{align}
on $M$ with $u=c_\pm$ on $\partial_\pm M$ and $v=d_\pm$ on $\partial_\pm M$, we will first obtain uniform estimates for the system
\begin{align}\label{eq:epsilon}
  \begin{split}
\Delta u_{\sigma,\varepsilon}=&-\sigma\tr_g(k)|\nabla u_{\sigma,\varepsilon}|+\frac{3|\nabla u_{\sigma,\varepsilon}||\nabla v_{\sigma,\varepsilon}|+\langle \nabla u_{\sigma,\varepsilon},\nabla v_{\sigma,\varepsilon}\rangle}{|u_{\sigma,\varepsilon}+v_{\sigma,\varepsilon}|+\varepsilon},\\
\Delta v_{\sigma,\varepsilon}=&\sigma\tr_g(k)|\nabla v_{\sigma,\varepsilon}|+\frac{3|\nabla u_{\sigma,\varepsilon}||\nabla v_{\sigma,\varepsilon}|+\langle \nabla u_{\sigma,\varepsilon},\nabla v_{\sigma,\varepsilon}\rangle}{|u_{\sigma,\varepsilon}+v_{\sigma,\varepsilon}|+\varepsilon}
\end{split}
\end{align}
with Dirichlet boundary conditions
\begin{align}\label{eq:epsilon2}
    \begin{split}
         u_{\sigma,\varepsilon}=&c_\pm\quad\quad\text{and}\quad\quad  v_{\sigma,\varepsilon}=\sigma d_\pm+(1-\sigma)c_\pm-\varepsilon \quad\quad\text{on $\partial_\pm M$}.
    \end{split}
\end{align}
Here $\sigma\in[0,1]$, and $\varepsilon>0$ is sufficiently small such that $\sigma d_\pm+(1-\sigma)c_\pm-\varepsilon>0$ for all $\sigma\in[0,1]$.
Without loss of generality we assume that $c_-< c_+$ and $d_-<d_+$.

\begin{lemma}\label{mp lemma}
Suppose $u_{\sigma,\varepsilon}\in C^{2,\alpha}(M)$ and $v_{\sigma,\varepsilon}\in C^{2,\alpha}(M)$ solve system \eqref{eq:epsilon}, \eqref{eq:epsilon2}. Then we have
\begin{align*}
   c_-\le u_{\sigma,\varepsilon}\le c_+,\quad\quad\quad\quad \sigma d_-+(1-\sigma)c_--\varepsilon\le v_{\sigma,\varepsilon}\le \sigma d_++(1-\sigma)c_+-\varepsilon
\end{align*}
on $M$.
\end{lemma}
 
 \begin{proof}
This follows immediately from the maximum principle for elliptic equations.
 \end{proof}
 
 \begin{lemma}
 Suppose $u_{\sigma,\varepsilon}\in C^{2,\alpha}(M)$ and $v_{\sigma,\varepsilon}\in C^{2,\alpha}(M)$ solve system \eqref{eq:epsilon}, \eqref{eq:epsilon2}.
 Then there exists a constant $C$ independent of $\sigma, u_\varepsilon,v_\varepsilon,\varepsilon$ such that 
 \begin{align*}
     \|u_{\sigma,\varepsilon}\|_{W^{2,p}(M)}+\|v_{\sigma,\varepsilon}\|_{W^{2,p}(M)}\le C.
 \end{align*}
 \end{lemma}
 
 \begin{proof}
 To prove this statement, it will be helpful to rewrite the above system in terms of 
 \begin{align*}
     w_{\sigma,\varepsilon}=v_{\sigma,\varepsilon}-u_{\sigma,\varepsilon},\quad\quad \text{and}\quad\quad h_{\sigma,\varepsilon}=\frac1{u_{\sigma,\varepsilon}+v_{\sigma,\varepsilon}+\epsilon}.
 \end{align*} 
We compute
 \begin{align}\label{w eq}
     \Delta w_{\sigma,\varepsilon}=\sigma\tr_g(k)(|\nabla u_{\sigma,\varepsilon}|+|\nabla v_{\sigma,\varepsilon}|),
 \end{align}
and 
 \begin{align*}
   \frac12 h_{\sigma,\varepsilon}^{-2}\Delta h_{\sigma,\varepsilon}=&-\frac12\Delta (u_{\sigma,\varepsilon}+v_{\sigma,\varepsilon}+\epsilon)+\frac{|\nabla (u_{\sigma,\varepsilon}+v_{\sigma,\varepsilon}+\epsilon)|^2}{u_{\sigma,\varepsilon}+v_{\sigma,\varepsilon}+\epsilon}\\
   =&\frac1{u_{\sigma,\varepsilon}+v_{\sigma,\varepsilon}+\epsilon}(3|\nabla u_{\sigma,\varepsilon}|^2-3|\nabla u_{\sigma,\varepsilon}||\nabla u_{\sigma,\varepsilon}+w_{\sigma,\varepsilon}|+3\langle \nabla u_{\sigma,\varepsilon},\nabla w_{\sigma,\varepsilon}\rangle +|\nabla w_{\sigma,\varepsilon}|^2)\\
   &+\sigma\frac {\tr_g(k)}2|\nabla u_{\sigma,\varepsilon}|-\sigma\frac {\tr_g(k)}2|\nabla u_{\sigma,\varepsilon}+w_{\sigma,\varepsilon}|.
 \end{align*}
 Using the identity 
 \begin{align*}|\nabla (u_{\sigma,\varepsilon}+w_{\sigma,\varepsilon})|-|\nabla u_{\sigma,\varepsilon}|=\frac1{|\nabla (u_{\sigma,\varepsilon}+w_{\sigma,\varepsilon})|+|\nabla u_{\sigma,\varepsilon}|}(|\nabla w_{\sigma,\varepsilon}|^2+2\langle \nabla u_{\sigma,\varepsilon},\nabla w_{\sigma,\varepsilon}\rangle),
 \end{align*}
 we obtain
  \begin{align}\label{h eq}
  \begin{split}
   &\frac12 h_{\sigma,\varepsilon}^{-2}\Delta h_{\sigma,\varepsilon}  \\
   =&-\frac1{u_{\sigma,\varepsilon}+v_{\sigma,\varepsilon}+\epsilon}\frac{3|\nabla u_{\sigma,\varepsilon}|}{|\nabla (u_{\sigma,\varepsilon}+v_{\sigma,\varepsilon})|+|\nabla u_{\sigma,\varepsilon}|}(|\nabla w_{\sigma,\varepsilon}|^2+2\langle \nabla u_{\sigma,\varepsilon},\nabla w_{\sigma,\varepsilon}\rangle)\\
   &+\frac1{u_{\sigma,\varepsilon}+v_{\sigma,\varepsilon}+\epsilon}\left(3\langle \nabla u_{\sigma,\varepsilon},\nabla w_{\sigma,\varepsilon}\rangle +|\nabla w_{\sigma,\varepsilon}|^2\right)\\
   &-\sigma\frac {\tr_g(k)}2\frac1{|\nabla (u_{\sigma,\varepsilon}+w_{\sigma,\varepsilon})|+|\nabla u_{\sigma,\varepsilon}|}(|\nabla w_{\sigma,\varepsilon}|^2+2\langle \nabla u_{\sigma,\varepsilon},\nabla w_{\sigma,\varepsilon}\rangle)\\
   =&\frac1{u_{\sigma,\varepsilon}+v_{\sigma,\varepsilon}+\epsilon}\left(-3|\nabla u_{\sigma,\varepsilon}|\frac1{|\nabla (u_{\sigma,\varepsilon}+w_{\sigma,\varepsilon})|+|\nabla u_{\sigma,\varepsilon}|}|\nabla w_{\sigma,\varepsilon}|^2 +|\nabla w_{\sigma,\varepsilon}|^2\right)\\
   &+\frac1{u_{\sigma,\varepsilon}+v_{\sigma,\varepsilon}+\epsilon}3\frac{\langle \nabla u_{\sigma,\varepsilon},\nabla w_{\sigma,\varepsilon}\rangle}{(|\nabla (u_{\sigma,\varepsilon}+w_{\sigma,\varepsilon})|+|\nabla u_{\sigma,\varepsilon}|)^2}(|\nabla w_{\sigma,\varepsilon}|^2+2\langle \nabla u_{\sigma,\varepsilon},\nabla w_{\sigma,\varepsilon}\rangle) \\
   &-\sigma\frac {\tr_g(k)}2\frac1{|\nabla (u_{\sigma,\varepsilon}+w_{\sigma,\varepsilon})|+|\nabla u_{\sigma,\varepsilon}|}(|\nabla w_{\sigma,\varepsilon}|^2+2\langle \nabla u_{\sigma,\varepsilon},\nabla w_{\sigma,\varepsilon}\rangle).
   \end{split}
 \end{align}
 Next, we estimate equation \eqref{w eq}
  \begin{align}\label{eq:w}
  \begin{split}
     |\Delta w_{\sigma,\varepsilon}|\le& C|\nabla (u_{\sigma,\varepsilon}+v_{\sigma,\varepsilon})|+C|\nabla (u_{\sigma,\varepsilon}-v_{\sigma,\varepsilon})|\\
     =& C|\nabla h_{\sigma,\varepsilon}|h_{\sigma,\varepsilon}^{-2}+C|\nabla w_{\sigma,\varepsilon}|\\
     \le &
     C|\nabla h_{\sigma,\varepsilon}|+C|\nabla w_{\sigma,\varepsilon}|
     \end{split}
 \end{align}
 where $C$ is depending on $(M,g,k),c_\pm,d_\pm$, and whose value may change from line to line.
Moreover, using equation \eqref{h eq} and Lemma \ref{mp lemma} 
 \begin{align}\label{eq:h}
     |\Delta h_{\sigma,\varepsilon}|\le & Ch_{\sigma,\varepsilon}^{3}|\nabla w_{\sigma,\varepsilon}|^2+Ch_{\sigma,\varepsilon}^{2}|\nabla w_{\sigma,\varepsilon}|
     \le 
     C|\nabla w_{\sigma,\varepsilon}|^2+C
 \end{align}
 where we used 
 \begin{align*}
     \frac{|\nabla w_{\sigma,\varepsilon}|}{|\nabla (u_{\sigma,\varepsilon}+w_{\sigma,\varepsilon})|+|\nabla u_{\sigma,\varepsilon}|}\le1.
 \end{align*}
According to the $W^{2,p}$ estimate for elliptic equations,
 \begin{align*}
     \|w_{\sigma,\varepsilon}\|_{W^{2,p}(M)}\le  C+C\|h_{\sigma,\varepsilon}\|_{W^{1,p}(M)}
 \end{align*}
 and
  \begin{align*}
     \|h_{\sigma,\varepsilon}\|_{W^{2,p}(M)}\le C+C\|w_{\sigma,\varepsilon}\|^2_{W^{1,2p}(M)}.
 \end{align*}
By the Gagliardo-Nirenberg interpolation inequality and the $L^\infty$-bounds for $w_{\sigma, \varepsilon},  h_{\sigma,\varepsilon}$, we have
 \begin{align*}
     \|\nabla w_{\sigma,\varepsilon}\|^2_{L^{2p}(M)}\le C\|\nabla^2w_{\sigma,\varepsilon}\|_{L^p(M)}\|w_{\sigma,\varepsilon}\|_{L^\infty(M)}+C\|w_{\sigma,\varepsilon}\|_{L^\infty(M)}\le C\|\nabla^2w_{\sigma,\varepsilon}\|_{L^p(M)}+C
 \end{align*}
and
 \begin{align*}
     \|\nabla h_{\sigma,\varepsilon}\|_{L^{p}(M)}\le C\|\nabla^2h_{\sigma,\varepsilon}\|_{L^p(M)}^\delta\|h_{\sigma,\varepsilon}\|_{L^\infty(M)}^{1-\delta}+C\|w_{\sigma,\varepsilon}\|_{L^\infty(M)}\le C\|\nabla^2h_{\sigma,\varepsilon}\|_{L^p(M)}^\delta+C
 \end{align*}
where 
\begin{align*}
    \delta
    =\frac{p-3}{2p-3}<1.
\end{align*}
Consequently,
  \begin{align*}
     \|h_{\sigma,\varepsilon}\|_{W^{2,p}(M)}\le& C+C\|w_{\sigma,\varepsilon}\|_{W^{2,p}(M)}\\\le& C+C\|h_{\sigma,\varepsilon}\|_{W^{1,p}(M)}\\
     \le& C+C\|h_{\sigma,\varepsilon}\|_{W^{2,p}(M)}^{\delta}\le C+\frac12 \|h_{\sigma,\varepsilon}\|_{W^{2,p}(M)}.
 \end{align*}
Thus, we have $\|h_{\sigma,\varepsilon}\|_{W^{2,p}(M)}\le C$ which implies $\|w_{\sigma,\varepsilon}\|_{W^{2,p}(M)}\le C$.
Reconstructing $u_{\sigma,\varepsilon},v_{\sigma,\varepsilon}$ from $h_{\sigma,\varepsilon},w_{\sigma,\varepsilon}$, we also obtain $\|u_{\sigma,\varepsilon}\|_{W^{2,p}(M)}+\|v_{\sigma,\varepsilon}\|_{W^{2,p}(M)}\le C$ which finishes the proof.
 \end{proof}

We can use the Sobolev inequality and Schauder estimates to improve the above estimate to a $C^{2,\alpha}$ estimate.
More precisely, we obtain:

\begin{lemma}\label{prop:PDE estimate}
Suppose $u_{\sigma,\varepsilon}\in C^{2,\alpha}(M)$ and $v_{\sigma,\varepsilon}\in C^{2,\alpha}(M)$ solve system \eqref{eq:epsilon}, \eqref{eq:epsilon2}.
Then there exists a constant $C$ depending on $(M,g,k),c_\pm, d_\pm$, but independent of $\sigma, u_{\sigma,\varepsilon},v_{\sigma,\varepsilon},\varepsilon$ such that
\begin{align*}
    \|u_{\sigma,\varepsilon}\|_{C^{2,\alpha}(M)}+\|v_{\sigma,\varepsilon}\|_{C^{2,\alpha}(M)}\le C.
\end{align*}
\end{lemma}

Having obtained uniform estimates for system \eqref{eq:epsilon} we will use Leray-Schauder's fixed point theorem below to obtain solutions of \eqref{eq:epsilon} for $\sigma=1$.
Passing to a limit $\varepsilon\to0$ then gives a solution to \eqref{eq:main} establishing Theorem \ref{T:PDE}.

 \begin{proof}[Proof of Theorem \ref{T:PDE}]
 Let $\phi_\pm^{\sigma,\varepsilon}$ be two smooth functions on $M$ with $\phi_\pm^{\sigma,\varepsilon}=c_\pm$, $\phi_\pm^{\sigma,\varepsilon}=\sigma d_\pm+(1-\sigma)c_\pm-\varepsilon$ on $\partial_\pm M$.
 We denote with $C^{2,\alpha}_0(M)$ the set $C^{2,\alpha}$-functions on $M$ which vanish on $\partial_\pm M$.
 Observe that $C^{2,\alpha}_0(M)$ is a Banach space.
 We define a family of maps $\mathcal F_{\sigma,\epsilon}:C^{2,\alpha}_0(M)\oplus C^{2,\alpha}_0(M)\to C^{2,\alpha}_0(M)\oplus C^{2,\alpha}_0(M)$ via
\begin{align*}
    \mathcal F_{\sigma,\varepsilon}(u,v)
  =    
 \left[\Delta_0^{-1}\left(\mathcal G^{\sigma,\varepsilon}_-(u,v)\right), \Delta_0^{-1}\left(\mathcal G^{\sigma,\varepsilon}_+(u,v)\right)                   \right]
\end{align*}
 where
\begin{align*}
    \mathcal G^{\sigma,\varepsilon}_\pm(u,v)=&\pm\sigma\tr_g(k)|\nabla (u+\phi_-)|\\
    &+\frac{3|\nabla (u+\phi_-)||\nabla (v+\phi_+)|+\langle \nabla (u+\phi_-),\nabla (v+\phi_+)\rangle}{|u+\phi_-+v+\phi_+|+\varepsilon}-\Delta \phi_\pm^{\sigma,\varepsilon}
\end{align*}
and where $\Delta^{-1}_0$ maps a function $f$ to the solution $\psi$ of $\Delta \psi=f$ on $M$ with vanishing Dirichlet boundary data.
By standard elliptic theory, $\mathcal F_{\sigma,\varepsilon}$ is indeed a map into $C^{2,\alpha}_0(M)\oplus C^{2,\alpha}_{0}(M)$.
Moreover, $\mathcal F_{\sigma,\varepsilon}$ is a compact operator.
More precisely, let $\{u_i,v_i\}$ be a bounded sequence in $C^{2,\alpha}_0(M)$.
Then $\mathcal G^{\sigma,\varepsilon}_\pm(u_i,v_i)$ are bounded sequences in $C^{0,\beta}_0(M)$ for any fixed $\beta\in(\alpha,1)$.
This implies that $\Delta_0^{-1}(\mathcal G^{\sigma,\varepsilon}_\pm(u_i,v_i))$ are bounded sequences in $C^{2,\beta}_0(M)$ which subsequentially converge in $C^{2,\alpha}_0(M)$ by Arzel\`a-Ascoli.

Observe that if $\mathcal F_{\sigma,\varepsilon}(u,v)=(u,v)$, then $(u+\phi_-,v+\phi_+)$ solve system \eqref{eq:epsilon}.
Next, we show that there is a solution $(u_{0,\varepsilon},v_{0,\varepsilon})$ to the equation $\mathcal F_{0,\varepsilon}(u_{0,\varepsilon},v_{0,\varepsilon})=(u_{0,\varepsilon},v_{0,\varepsilon})$.

Let $U_{0,\varepsilon}$ be the harmonic function with $U_{0,\varepsilon}=\frac1{c_\pm}$ on $\partial_\pm M$.
Then $u_{0,\varepsilon}=\frac1{U_{0,\varepsilon}}$ satisfies
\begin{align*}
    \Delta u_{0,\varepsilon}=2\frac{|\nabla u_{0,\varepsilon}|^2}{u_{0,\varepsilon}}
\end{align*}
with $u_{0,\varepsilon}=c_\pm $ on $\partial_\pm M$.
Next, let $v_{0,\varepsilon}=u_{0,\varepsilon}-\varepsilon$, and note that $v_{0,\varepsilon}=c_\pm-\varepsilon$ on $\partial_\pm M$ and $v_{0,\varepsilon}>0$ on $M$.
Moreover,
\begin{align*}
  \begin{split}
\Delta u_{0,\varepsilon}=&\frac{3|\nabla u_{0,\varepsilon}||\nabla v_{0,\varepsilon}|+\langle \nabla u_{0,\varepsilon},\nabla v_{0,\varepsilon}\rangle}{|u_{0,\varepsilon}+v_{0,\varepsilon}|+\varepsilon},\\
\Delta v_{0,\varepsilon}=&\frac{3|\nabla u_{0,\varepsilon}||\nabla v_{0,\varepsilon}|+\langle \nabla u_{0,\varepsilon},\nabla v_{0,\varepsilon}\rangle}{|u_{0,\varepsilon}+v_{0,\varepsilon}|+\varepsilon}.
\end{split}
\end{align*}

Hence, we can use the uniform estimates, Lemma \ref{prop:PDE estimate}, together with Leray-Schauder's fixed point theorem, see for instance Theorem 11.6 in \cite{GilbargTrudinger}, to deduce that there exists a solution of $\mathcal F_{1,\varepsilon}(u_{1,\varepsilon},v_{1,\varepsilon})=(u_{1,\varepsilon},v_{1,\varepsilon})$.
The functions $u_{1,\varepsilon},v_{1,\varepsilon}$ are uniformly bounded away from zero, and we have $C^{2,\alpha}$-estimates for $(u_{1,\varepsilon},v_{1,\varepsilon})$ which are uniform in terms of $\varepsilon$.
Consequently, we can take a subsequential limit $(u_{1,\varepsilon},v_{1,\varepsilon})$ as $\varepsilon\to0$ to obtain solutions $(u,v)\in C^{2,\gamma}$, $\gamma<\alpha$, to system \eqref{eq:main}.
 \end{proof}

A crucial ingredient of the above proof is that the system \eqref{system} takes a simpler form for $w=u-v$ and $h=\frac1{u+v}$ as in \eqref{eq:w} and \eqref{eq:h}.
We remark that for the $p$-harmonic system, i.e. system \eqref{system} for $a\in(0,1)$, can be rewritten in a very similar form to \eqref{eq:w} and \eqref{eq:h} though we have to re-define $h=(u+v)^{-\frac{1+a}{1-a}}$.
Note that the radial function $r^{-\frac{1+a}{1-a}}$ is $p$-harmonic in $\R^3$ for $p=2-a$.
We also expect that the solutions of system \eqref{system} for $a=0$ established in this section can be used to give a new proof for the spacetime PMT including a new positive lower bound for the mass.



\begin{thebibliography}{99}

\bibitem{AMMO}
V. Agostiniani, C. Mantegazza, L. Mazzieri and F. Oronzio, \textit{Riemannian Penrose Inequality via nonlinear potential theory}, arXiv preprint arXiv:2205.11642 (2022).

\bibitem{AMO}
V. Agostianini, L. Mazzieri and F. Oronzio, \textit{A Green's function proof of the Positive Mass Theorem}, preprint, 2021, arXiv:2108.08402

\bibitem{BartnikChrusciel}
R. Bartnik, and P. Chrusciel, \textit{Boundary value problems for Dirac-type equations}, J. reine angew. Math., 2005.

\bibitem{Ben-Dov}
 I. Ben-Dov, \textit{The Penrose inequality and apparent horizons}, Phys. Rev. D 70, 124031.
(2004).

\bibitem{Bray} H. Bray, \textit{Proof of the Riemannian Penrose inequality using the positive mass theorem.}, J. Differential Geom., \textbf{59} (2001), no. 2, 177-267.

\bibitem{BHMS}
H. Bray, S. Hayward, M. Mars and W. Simon, \textit{Generalized inverse mean curvature flows in spacetime}, Communications in mathematical physics 272.1 (2007): 119-138.

\bibitem{BHKKZ} H. Bray, S. Hirsch, D. Kazaras, M. Khuri, and Y. Zhang, \textit{Spacetime harmonic functions and application to mass}, Perspectives in scalar curvature, 2022.

\bibitem{BHKKZ2} H. Bray, S. Hirsch, D. Kazaras, M. Khuri, and Y. Zhang, \textit{Spacetime harmonic functions on asymptotically hyperbolic manifolds and mass}, in preparation.

\bibitem{BKKS} H. Bray, D. Kazaras, M. Khuri, and D. Stern, \textit{Harmonic functions and the mass of 3-dimensional asymptotically flat Riemannian manifolds}, preprint, 2019. arXiv:1911.06754

\bibitem{BrayKhuri} H. Bray and M. Khuri, \textit{Pde's which imply the Penrose conjecture}, Asian Journal of Mathematics 15.4 (2011): 557-610.

\bibitem{BrayKhuri2}
H. Bray and M. Khuri, \textit{A Jang equation approach to the Penrose inequality}, arXiv preprint arXiv:0910.4785 (2009).

\bibitem{BrayLee} H. Bray and D. Lee, \textit{On the Riemannian Penrose inequality in dimensions less than eight}. Duke Mathematical Journal 148.1 (2009): 81-106.

\bibitem{CarrascoMars}
A. Carrasco and M. Mars, \textit{A counterexample to a recent version of the Penrose conjecture}, Classical and Quantum Gravity 27.6 (2010): 062001.

\bibitem{Christodoulou}
D. Christodoulou, \textit{Examples of naked singularity formation in the gravitational collapse of a scalar field} Annals of Mathematics 140.3 (1994): 607-653.

\bibitem{Christodoulou2}
D. Christodoulou, \textit{The instability of naked singularities in the gravitational collapse of a scalar field} Annals of Mathematics 149.1 (1999): 183-217.

\bibitem{Eichmair} M. Eichmair, \textit{The Jang equation reduction of the spacetime positive energy theorem in dimensions less than eight}, Comm. Math. Phys., \textbf{319} (2013), no. 3, 575-593.

\bibitem{EHLS} M. Eichmair, L.-H. Huang, D. Lee, and R. Schoen, \textit{The spacetime positive mass theorem in dimensions less than eight}, J. Eur. Math. Soc. (JEMS), \textbf{18} (2016), no. 1, 83-121.

\bibitem{Frauendiener} 
J. Frauendiener, \textit{On the Penrose inequality}, Physical review letters 87.10 (2001): 101101.

\bibitem{Geroch}
R.Geroch, \textit{Energy extraction}, Ann.New York Acad.Sci. 224 (1973) 108–17.

\bibitem{GilbargTrudinger}
D. Gilbarg and N. Trudinger, \textit{Elliptic partial differential equations of second order}, Vol. 224. No. 2. Berlin: springer, 1977.

\bibitem{Hawking}
S. Hawking, \textit{Black holes in general relativity}, Communications in Mathematical Physics 25.2 (1972): 152-166.

\bibitem{Hayard}
S. Hayward, \textit{Gravitational energy in spherical symmetry}, Phys. Rev. D, 53
(1996), 1938-1949.

\bibitem{Hayard2}
S. Hayward, \textit{Inequalities relating area, energy, surface gravity, and charge of black
holes}, Phys. Rev. Lett., 81 (1998), 4557-4559.

\bibitem{HKK} S. Hirsch, D. Kazaras, M. Khuri, \textit{Spacetime Harmonic Functions and the Mass of 3-Dimensional Asymptotically Flat Initial Data for the Einstein Equations}, J. Differential Geom., to appear, arXiv:2002.01534.

\bibitem{HKKZ}
S. Hirsch, D. Kazaras, M. Khuri and Y. Zhang, \textit{Scalar and Ricci curvature rigidity of spheres and bands}, in preparation.

\bibitem{HMT} S. Hirsch, P. Miao, and T.-Y. Tsang. \textit{Mass of asymptotically flat 3-manifolds with boundary.}, preprint, 2020, arXiv:2009.02959.
 
\bibitem{HirschZhang}
S. Hirsch and Y. Zhang, \textit{The case of equality for the spacetime positive mass theorem}, preprint, 2022, arXiv:2203.01984.

\bibitem{HirschZhang2}
S. Hirsch and Y. Zhang, \emph{Initial data sets with vanishing mass are contained in pp-wave spacetimes}, arXiv preprint arXiv:2403.15984 (2024).

\bibitem{HuiskenIlmanen} G. Huisken and T. Ilmanen, \textit{The inverse mean curvature flow and the Riemannian Penrose inequality}, J. Differential Geom., \textbf{59} (2001), 353-437.

\bibitem{HuiskenWolf} 
G. Huiksen and M. Wolf, \textit{On the evolution of hypersurfaces along their inverse spacetime mean curvature}, preprint, 2022, arXiv:2208.05709.

\bibitem{IMM}
M. Iriondo, E. Malec, and N. O Murchadha, \textit{Constant mean curvature slices and
trapped surfaces in asymptotically flat spherical spacetimes}, Phys. Rev. D, 54
(1996), 4792-4798.

\bibitem{JangWald}
P.S. Jang and R.M. Wald, \textit{The Positive Energy Conjecture and the Cosmic Censor
Hypothesis}, J.Math.Phys. 18 (1977) 41–44.

\bibitem{Jaracz}
J.S. Jaracz, \emph{Nonexistence of solutions to the coupled generalized Jang equation/zero divergence system}, Classical and Quantum Gravity 40.19 (2023): 195013.

\bibitem{KWY}
M. Khuri, G. Weinstein and S. Yamada, \textit{Proof of the Riemannian Penrose inequality with charge for multiple black holes}, Journal of Differential Geometry 106.3 (2017): 451-498.

\bibitem{Krohn}
J. Kroon, \textit{Characterization of Schwarzschildean initial data}, Physical Review D 72.8 (2005): 084003.

\bibitem{Lee} D. Lee, \textit{Geometric Relativity}, Graduate Studies in Mathematics, Volume \textbf{201}, 2019.

\bibitem{Li} Y. Li, \textit{Ricci flow on asymptotically Euclidean manifolds}, Geometry $\&$ Topology, \textbf{22} (2018), no. 3, 1837-1891.

\bibitem{MalecMurchadha}
 E. Malec, and N. O Murchadha, \textit{Trapped surfaces and the Penrose inequality in
spherically symmetric geometries}, Phys. Rev. D, 49 (1994), 6931-6934.

\bibitem{Mars} M. Mars, \textit{Present status of the Penrose inequality}, Classical and Quantum Gravity 26.19 (2009): 193001.

\bibitem{Miao}
P. Miao, \textit{Mass, capacitary functions, and positive mass theorems with or without boundary}, preprint, 2022, arXiv:2207.03467.

\bibitem{Moore}
K. Moore, \textit{On the evolution of hypersurfaces by their inverse null mean curvature}, Journal of Differential Geometry 98.3 (2014): 425-466.

\bibitem{Moser}
R. Moser, \textit{The inverse mean curvature flow and p-harmonic functions}, Journal of the European Mathematical Society 9.1 (2007): 77-83.

\bibitem{Penrose2}
R. Penrose, \textit{Gravitational collapse and space-time singularities}, Physical Review Letters 14.3 (1965): 57.

\bibitem{Penrose}
R. Penrose, \textit{Naked singularities}, Ann. N. Y. Acad. Sci. 224, 125-134 (1973).

\bibitem{Roesch}
H. Roesch, \textit{Proof of a null penrose conjecture using a new quasi-local mass}, preprint, arXiv:1609.02875 (2016).

\bibitem{Stern} D. Stern, \textit{Scalar curvature and harmonic maps to $S^1$}, J. Differential Geom., to appear. arXiv:1908.09754


\bibitem{SY1} R. Schoen, and S.-T. Yau, \textit{On the proof of the positive mass conjecture in general relativity}, Comm. Math. Phys., \textbf{65} (1979), no. 1, 45-76.

\bibitem{SY2} R. Schoen, and S.-T. Yau, \textit{Proof of the positive mass theorem II}, Comm. Math. Phys., \textbf{79} (1981), 231-260.

\bibitem{Wald}
R. Wald, \textit{General relativity}, University of Chicago press, 2010.

\bibitem{Witten} E. Witten, \textit{A simple proof of the positive energy theorem}, Comm. Math. Phys., \textbf{80} (1981), no. 3, 381-402.

\end{thebibliography}
\end{document}